\documentclass[11pt]{article}

\usepackage[hmargin={27mm,27mm},vmargin={25mm,25mm}]{geometry}
\usepackage{amsfonts}
\usepackage{amssymb}
\usepackage{amsthm}
\usepackage{bm}
\usepackage{mathtools}
\usepackage{latexsym}
\usepackage{graphicx}
\usepackage{booktabs}
\usepackage{pgfplots,pgfplotstable}
\usepackage{xcolor}
\usepackage{setspace}
\usepackage{url}
\usepackage{float}
\usepackage{subcaption}
\usepackage{cancel}

\usepackage[colorlinks,allcolors={blue},linkcolor={black}]{hyperref}

\numberwithin{equation}{section}

\usepackage{authblk}
\newcommand{\email}[1]{\href{mailto:#1}{#1}}

\newtheorem{theorem}{Theorem}

\newtheorem{lemma}[theorem]{Lemma}
\newtheorem{corollary}[theorem]{Corollary}

\newtheorem{assumption}{Assumption}

\newtheorem{remark}{Remark}

\newcommand{\bmrm}[1]{{\bm{{\rm #1}}}}

\newcommand{\mat}[1]{\bmrm{#1}}

\newcommand{\bmn}{{\bm{n}}}

\newcommand{\bmU}{\bm{U}}
\newcommand{\bmF}{\bm{F}}

\newcommand{\Hmin}[1]{H_{#1,{\rm min}}}
\newcommand{\Hmax}[1]{H_{#1,{\rm max}}}

\newcommand{\bmtau}{{\bm{\tau}}}

\newcommand{\calF}{\mathcal{F}}

\newcommand{\calH}{\mathcal{H}}

\newcommand{\calL}{\mathcal{L}}
\newcommand{\calM}{\mathcal{M}}

\newcommand{\calS}{\mathcal{S}}
\newcommand{\calT}{\mathcal{T}}

\newcommand{\bbN}{\mathbb{N}}

\newcommand{\bbP}{\mathbb{P}}

\newcommand{\bbR}{\mathbb{R}}

\newcommand{\rma}{{\rm{a}}}
\newcommand{\rmb}{{\rm{b}}}

\newcommand{\rmi}{{\rm{i}}}

\newcommand{\rmp}{{\rm{p}}}

\newcommand{\rms}{{\rm{s}}}

\newcommand{\rmA}{{\rm{A}}}

\newcommand{\rmL}{{\rm{L}}}

\newcommand{\eq}{ ={}& }
\newcommand{\lea}{ \le{}& }

\newcommand{\les}{ \lesssim{}& }

\newcommand{\nn}{\nonumber}
\newcommand{\nl}{\nn\\}

\newcommand{\defeq}{\vcentcolon=}

\newcommand{\ul}[1]{\underline{#1}}
\newcommand{\ol}[1]{\overline{#1}}

\newcommand{\bdry}{\partial}

\newcommand{\Mh}[1][h]{\calM_{#1}}
\newcommand{\Th}[1][h]{\calT_{#1}}
\newcommand{\Fh}[1][h]{\calF_{#1}}
\newcommand{\Fhb}{\Fh^{\rmb}}
\newcommand{\Fhi}{\Fh^{\rmi}}

\newcommand{\T}{{T}}
\newcommand{\F}{{F}}
\newcommand{\FT}{{\bdryT}}
\newcommand{\FTi}{{\bdryTi}}
\newcommand{\FTs}{{\bdryTs}}

\newcommand{\hT}{h_\T}
\newcommand{\hTi}{h_{\T_i}}
\newcommand{\hTs}{h_{\T_\star}}
\newcommand{\hF}{h_\F} 
\newcommand{\hX}{h_X}
\newcommand{\hmin}{h_{\rm{min}}} 
\newcommand{\hmax}{h_{\rm{max}}}

\newcommand{\bdryT}{{\bdry \T}}
\newcommand{\bdryTi}{{\bdry \T_i}}
\newcommand{\bdryTs}{{\bdry \T_\star}}

\newcommand{\nor}{\bmn}
\newcommand{\norT}{\nor_{\bdryT}}

\newcommand{\norTF}{\nor_{\T\F}}

\def\R{\bbR}

\def\N{\bbN}

\newcommand{\POLY}[1]{\bbP^{#1}}
\newcommand{\HS}[1]{H^{#1}}
\newcommand{\HONE}{\HS{1}}
\newcommand{\HONEzr}{\HONE_0}

\newcommand{\LP}[1]{L^{#1}}
\newcommand{\LTWO}{\LP{2}}

\newcommand{\norm}[2][]{\|#2\|_{#1}}
\newcommand{\seminorm}[2][]{|#2|_{#1}}
\newcommand{\brac}[2][]{(#2)_{#1}}

\newcommand{\Brac}[2][]{\Big(#2\Big)_{#1}}

\newcommand{\energynorm}[1]{\norm[\rma, h]{#1}}

\newcommand{\piT}[1]{\pi_T^{#1}}
\newcommand{\piTzr}[1]{\piT{0, #1}}

\newcommand{\piFT}[1]{\pi_\FT^{#1}}
\newcommand{\piFTzr}[1]{\piFT{0, #1}}

\newcommand{\piF}[1]{\pi_F^{#1}}
\newcommand{\piFzr}[1]{\piF{0, #1}}

\newcommand{\piX}[1]{\pi_X^{#1}}

\newcommand{\pT}[1]{\rmp_T^{#1}}
\newcommand{\pTi}[1]{\rmp_{T_i}^{#1}}
\newcommand{\pTs}[1]{\rmp_{T_\star}^{#1}}

\newcommand{\ph}[1]{\rmp_h^{#1}}

\newcommand{\U}[2]{\ul{U}_{#1}^{#2}}

\newcommand{\UT}[1]{\U{T}{#1}}

\newcommand{\UTkl}{\UT{k,l}}

\newcommand{\Uh}[1]{\U{h}{#1}}

\newcommand{\Uhkl}{\Uh{k,l}}
\newcommand{\Uhklzr}{\U{h, 0}{k,l}}

\newcommand{\I}[2]{\ul{I}_{#1}^{#2}}

\newcommand{\IT}[1]{\I{T}{#1}}

\newcommand{\ITkl}{\IT{k,l}}

\newcommand{\Ih}[1]{\I{h}{#1}}

\newcommand{\Ihkl}{\Ih{k,l}}

\def\a{\rma}
\newcommand{\aT}{\a_T}
\newcommand{\sT}{\rms_T}
\newcommand{\ah}{\a_h}
\newcommand{\sh}{\rms_h}

\newcommand{\sTgrad}{\sT^{\nabla}}
\newcommand{\sTbdry}{\sT^{\partial}}
\newcommand{\sTkminus}{\sT^{(k-1)}}

\newcommand{\vT}{v_T}
\newcommand{\vh}{v_h}
\newcommand{\vF}{v_F}
\newcommand{\vFT}{v_{\FT}}
\newcommand{\vFh}{v_{\Fh}}

\newcommand{\uT}{u_T}

\newcommand{\uFT}{u_{\FT}}
\newcommand{\uFh}{u_{\Fh}}

\newcommand{\wT}{w_T}

\newcommand{\wFT}{w_{\FT}}

\newcommand{\ulvT}{\ul{v}_T}
\newcommand{\ulvTi}{\ul{v}_{T_i}}
\newcommand{\ulvTs}{\ul{v}_{T_\star}}
\newcommand{\uluT}{\ul{u}_T}
\newcommand{\ulvh}{\ul{v}_h}
\newcommand{\uluh}{\ul{u}_h}

\newcommand{\ulwT}{\ul{w}_T}

\newcommand{\deltaT}[1]{\delta_T^{#1}}

\newcommand{\deltaFT}[1]{\delta_{\FT}^{#1}}
\newcommand{\deltaFTi}[1]{\delta_{\FTi}^{#1}}
\newcommand{\deltaFTs}[1]{\delta_{\FTs}^{#1}}

\newcommand{\gT}{g_\T}
\newcommand{\calST}{\calS_\T}
\newcommand{\Ah}{\rmA_h}
\newcommand{\Lh}{\rmL_h}

\newcommand{\logLogSlopeTriangle}[5]
{
	\pgfplotsextra
	{
		\pgfkeysgetvalue{/pgfplots/xmin}{\xmin}
		\pgfkeysgetvalue{/pgfplots/xmax}{\xmax}
		\pgfkeysgetvalue{/pgfplots/ymin}{\ymin}
		\pgfkeysgetvalue{/pgfplots/ymax}{\ymax}
		
		\pgfmathsetmacro{\xArel}{#1}
		\pgfmathsetmacro{\yArel}{#3}
		\pgfmathsetmacro{\xBrel}{#1-#2}
		\pgfmathsetmacro{\yBrel}{\yArel}
		\pgfmathsetmacro{\xCrel}{\xArel}
		
		\pgfmathsetmacro{\lnxB}{\xmin*(1-(#1-#2))+\xmax*(#1-#2)} 
		\pgfmathsetmacro{\lnxA}{\xmin*(1-#1)+\xmax*#1} 
		\pgfmathsetmacro{\lnyA}{\ymin*(1-#3)+\ymax*#3} 
		\pgfmathsetmacro{\lnyC}{\lnyA+#4*(\lnxA-\lnxB)}
		\pgfmathsetmacro{\yCrel}{\lnyC-\ymin)/(\ymax-\ymin)}
		
		\coordinate (A) at (rel axis cs:\xArel,\yArel);
		\coordinate (B) at (rel axis cs:\xBrel,\yBrel);
		\coordinate (C) at (rel axis cs:\xCrel,\yCrel);
		
		\draw[#5]   (A)-- node[pos=0.5,anchor=north] {\scriptsize{1}}
		(B)-- 
		(C)-- node[pos=0.,anchor=west] {\scriptsize{#4}} 
		cycle;
	}
}

\newcommand{\reverseLogLogSlopeTriangle}[5]
{
	\pgfplotsextra
	{
		\pgfkeysgetvalue{/pgfplots/xmin}{\xmin}
		\pgfkeysgetvalue{/pgfplots/xmax}{\xmax}
		\pgfkeysgetvalue{/pgfplots/ymin}{\ymin}
		\pgfkeysgetvalue{/pgfplots/ymax}{\ymax}
		
		\pgfmathsetmacro{\xArel}{#1}
		\pgfmathsetmacro{\yArel}{#3}
		\pgfmathsetmacro{\xBrel}{#1-#2}
		\pgfmathsetmacro{\yBrel}{\yArel}
		\pgfmathsetmacro{\xCrel}{\xBrel}
		
		\pgfmathsetmacro{\lnxB}{\xmin*(1-(#1-#2))+\xmax*(#1-#2)} 
		\pgfmathsetmacro{\lnxA}{\xmin*(1-#1)+\xmax*#1} 
		\pgfmathsetmacro{\lnyA}{\ymin*(1-#3)+\ymax*#3} 
		\pgfmathsetmacro{\lnyC}{\lnyA+#4*(\lnxA-\lnxB)}
		\pgfmathsetmacro{\yCrel}{\lnyC-\ymin)/(\ymax-\ymin)}
		
		\coordinate (A) at (rel axis cs:\xArel,\yArel);
		\coordinate (B) at (rel axis cs:\xBrel,\yBrel);
		\coordinate (C) at (rel axis cs:\xCrel,\yCrel);
		
		\draw[#5]   (A)-- node[pos=0.5,anchor=north] {\scriptsize{1}}
		(B)-- 
		(C) -- node[pos=0.,anchor=east] {\scriptsize{#4}}
		cycle;
	}
}

\usepackage{acronym}

\acrodef{xfem}[XFEM]{extended finite element method}
\acrodef{pde}[PDE]{partial differential equation}
\acrodef{fe}[FE]{finite element}
\acrodef{fem}[FEM]{finite element method}
\acrodef{vem}[VEM]{virtual element method}
\acrodef{hho}[HHO]{hybrid high-order}
\acrodef{hdg}[HDG]{hybridizable discontinuous Galerkin}
\acrodef{cg}[CG]{continuous Galerkin}
\acrodef{dg}[DG]{discontinuous Galerkin}

\begin{document}
	
\title{Conditioning of a Hybrid High-Order scheme on meshes with small faces}
\author[1]{Santiago Badia}
\author[2]{J\'er\^ome Droniou}
\author[3]{Liam Yemm}
\affil[1]{School of Mathematics, Monash University, Clayton, Victoria, 3800, Australia \& Centre Internacional 
\newline  de M\`etodes Num\`erics a l'Enginyeria, Barcelona, Spain
\email{santiago.badia@monash.edu}}
\affil[2]{School of Mathematics, Monash University, Clayton, Victoria, 3800, Australia, \email{jerome.droniou@monash.edu}}
\affil[3]{School of Mathematics, Monash University, Clayton, Victoria, 3800, Australia, \email{liam.yemm@monash.edu}}
\date{}
\maketitle 

\begin{abstract}
	We conduct a condition number analysis of a Hybrid High-Order (HHO) scheme for the Poisson problem. We find the condition number of the statically condensed system to be independent of the number of faces in each element, or the relative size between an element and its faces. The dependence of the condition number on the polynomial degree is tracked. Next, we consider HHO schemes on cut background meshes, which are commonly used in unfitted discretisations. It is well known that the linear systems obtained on these meshes can be arbitrarily ill-conditioned due to the presence of sliver-cut and small-cut elements. We show that the condition number arising from HHO schemes on such meshes is not as negatively effected as those arising from conforming methods. We describe how the condition number can be improved by aggregating ill-conditioned elements with their neighbours.
\end{abstract}


\section{Introduction}\label{sec:introduction}


Several hybrid discretisation methods have been proposed in recent years for the numerical discretisation of partial differential equations \cite{di-pietro.ern.ea:2014:arbitrary,beirao-da-veiga.brezzi.ea:2013:basic,Cockburn2009}. One of the selling points of these schemes is their geometrical flexibility. Discretisation spaces are not bound to specific element topologies and can readily be used on general polytopal meshes. Body-fitted unstructured mesh generation is one of the main bottlenecks in complex numerical simulations, which requires intensive human intervention. Usually, these meshes are composed of tetrahedral (and/or hexahedral) elements. Polytopal methods can provide sought-after flexibility in the mesh generation step.

In this work, we focus on the \ac{hho} method. Developed in \cite{di-pietro.ern.ea:2014:arbitrary, di-pietro.ern:2015:hybrid}, the HHO method is a modern polytopal method for elliptic PDEs. A key aspect of HHO is its applicability to generic meshes with arbitrarily shaped elements. Additionally, HHO methods are of arbitrary order, dimension independent, and are amenable to static condensation. We refer the reader to \cite{di-pietro.droniou:2020:hybrid} for a thorough review of the method and its applications. An analysis on skewed meshes has been carried out for a diffusion problem in \cite{droniou:2020:interplay} and identifies how the error estimate is impacted by the element distortion and local diffusion tensor. The recent work of \cite{droniou.yemm:2021:robust} shows the HHO method to be accurate on meshes possessing elements with arbitrarily many small faces. 

Unfitted (a.k.a.~embedded and immersed) discretisations can also simplify the geometrical discretisation step. The domain of interest is embedded in a simple background mesh (e.g., a Cartesian grid). The boundary (or interface) treatment is tackled at the numerical integration and discretisation step. 
Many unfitted \ac{fe} schemes that rely on a standard \ac{fe} space on the background mesh have been proposed;~see, e.g.~the \ac{xfem}~\cite{belytschko_arbitrary_2001}, the cutFEM method~\cite{burman_cutfem_2015}, the aggregated \ac{fem} ~\cite{badia.verdugo.ea:2018:aggregated}, the finite cell method~\cite{Schillinger2015} and \ac{dg} methods with element aggregation~\cite{johansson2013high}. We also make note of the reference \cite{beirao-da-veiga.canuto.ea:2021:equilibrium}, which uses a \ac{vem} to model a rigid leaflet submerged in a fluid and fixed to a rotational spring at one end. The thin leaflet `cuts' through an isotropic background mesh, thus requiring the model to be applied on cut meshes.

Unfitted formulations can produce arbitrarily ill-conditioned linear systems~\cite{de-prenter.erhoosel.ea:2017:condition}. The intersection of a background element with the physical domain can be arbitrarily small and with an unbounded aspect ratio. It is known as the \emph{small cut element problem}. This problem is also present on unfitted interfaces with a high contrast of physical properties~\cite{Neiva2021}. Few unfitted formulations are fully robust and optimal regardless of cut location or material contrast. The ill-conditioning issue was addressed in~\cite{burman2010ghost} via the so-called ghost penalty stabilisation. Instead of adding stabilisation terms, the small cut element problem can be fixed by element aggregation (or agglomeration). This approach has been proposed in~\cite{johansson2013high} for \ac{dg} methods. While aggregation is natural in \ac{dg} methods (these schemes can readily be used on polytopal meshes), its extension to conforming spaces is more involved. The design of well-posed $\mathcal{C}^{0}$ Lagrangian finite elements on agglomerated meshes has been proposed in \cite{badia.verdugo.ea:2018:aggregated}. The aggregated \ac{fem} constructs a discrete extension operator from well-posed to ill-posed degrees of freedom that preserves continuity. All these formulations enjoy good numerical properties, such as stability, condition number bounds, optimal convergence and continuity with respect to data.

The \ac{fe} discretisation of linear second-order elliptic operators (e.g. the Laplacian) in weak form produces linear systems such that the $\ell^2$-condition number (on shape regular, quasi-uniform meshes) scales as the inverse square of the mesh size \cite{Ern2006Jan}. Likewise, the condition number on regular triangular meshes of interior penalty Galerkin and local discontinuous Galerkin methods scale as the inverse square of the mesh size, whereas the condition number of non-symmetric DG methods can potentially scale sub-optimally \cite{castillo:2002:performance}. We refer to \cite{Cockburn2013} for the condition number analysis of \ac{hdg} methods on regular simplicial quasi-uniform meshes. The authors in \cite{Mascotto2018} investigate experimentally the ill-conditioning of the \ac{vem} for high-order bases on distorted meshes. In this work, we analyse the properties of the linear systems that arise from \ac{hho} formulations. We prove estimates for the condition number arising from such systems. Under general assumptions on the stabilisation term (allowing for standard choices of HHO stabilisation) and when $L^2$-orthonormal bases are chosen for face polynomial spaces, we show that the estimates remain robust with respect to small element faces and track the dependence in the estimates of the polynomial degree of the unknowns. 
The linear systems in HHO methods are obtained after the static condensation of the element unknowns.
This process allows the global system to depend only on the face unknowns \cite[Appendix B.3.2]{di-pietro.droniou:2020:hybrid}. In Section \ref{sec:eig.estimates} we state some estimates on the spectrum and conditioning of this condensed operator. We find that the condition number of the statically condensed system scales at worst like $\hmin^{-2}$ (where $\hmin$ denotes the minimum element diameter in the mesh) and that this bound is not affected by small faces. We also prove that if each face is attached (or close) to at least one element of diameter comparable to a characteristic mesh size $\hmax$, then the condition number scales as $\hmin^{-1}\hmax^{-1}$.
This sharper result is of practical interest when using cut meshes, since it is common to find  small cells on the boundary in touch with larger ones.
To the best of our knowledge, no condition number estimates on general meshes exist for \ac{hho} or \ac{vem}. We note that, given the links between HHO and other polytopal methods (see e.g.~\cite[Sec.~5.5.5]{di-pietro.droniou:2020:hybrid} for the relationship between \ac{hho} and non-conforming \ac{vem}, or \cite{Cockburn.Di-Pietro.ea:16} for the link \ac{hho}--\ac{hdg}), our results could easily be extended to such methods; more generally, the analysis of condition number we carry out here uses a rather general approach and would certainly extend to even more polytopal methods.

Next, we apply the \ac{hho} method on \emph{cut} meshes obtained by the intersection of cells in a background (usually Cartesian) mesh and the physical domain (represented as the interior of an oriented boundary representation, e.g., a surface mesh). The intersection is cell-wise and represents a sound alternative to unstructured mesh generation \cite{2110.01378}. The analysis tells us the potential conditioning issues of \ac{hho} schemes on such meshes, for which arbitrarily small elements (and faces) appear scattered among large elements \cite{burman.cicuttin.ea:2021:unfitted}. Based on the analysis, we know that we must aggregate highly distorted small cut elements (e.g., due to sliver cuts) to interior elements. Since arbitrary small faces do not affect condition number bounds, there is no need for face aggregation or stabilisation. This way, we end up with an \ac{hho} method on aggregated cut meshes that leads to well-posed linear systems and optimal condition numbers.

Hybrid methods on cut meshes have some benefits compared to more standard unfitted \acp{fe}. First, we can enforce Dirichlet boundary conditions strongly; there are degrees of freedom located on boundaries faces. In unfitted standard \acp{fe}, degrees of freedom are defined in the background mesh. Dirichlet boundary conditions and trace continuity on interfaces are weakly enforced (using, e.g., Nitsche's method \cite{hansbo2002unfitted}). Second, the method does not involve the tuning of additional stabilisation parameters, which can have an impact on results \cite{badia2021linking}. Third, the extension to high order is straightforward. It is more complicated in face-based ghost penalty (it involves penalty terms on jumps of high-order derivatives) \cite{burman2010ghost} or aggregated \acp{fe} (extension operators for high order can amplify rounding errors) \cite{badia.verdugo.ea:2018:aggregated}. 



The remainder of this paper is organised as follows: In Section \ref{sec:model.and.results} we introduce the HHO method and state our key findings. In Section \ref{sec:proofs} we prove the results and discuss viable stabilisation options, in Section \ref{sec:unfitted} we include a brief discussion of HHO on cut meshes, and in Section \ref{sec:numerical} we conduct a thorough numerical study of the condition number on various meshes.

\section{Presentation of the HHO method and main result}\label{sec:model.and.results}

\subsection{Model problem}

We take a polytopal domain \(\Omega\subset \R^d\), \(d\ge 2\) and a source term \(f\in \LTWO(\Omega)\), and consider the Dirichlet problem: find \(u\) such that 
\[
\begin{aligned}
	-\Delta u \eq f \quad\text{in}\quad\Omega,\\
	u \eq 0 \quad\text{on}\quad\partial\Omega.
\end{aligned}
\]
The variational problem reads: find \(u\in \HONEzr( \Omega) \) such that
\begin{equation}\label{eq:weak.form}
	\a(u, v)  = \calL(v), \qquad\forall v\in \HONEzr(\Omega), 
\end{equation}
where \(\a(u, v) \defeq \brac[\Omega]{\nabla u, \nabla v}\) and \(\calL(v) \defeq \brac[\Omega]{f, v}\). Here and in the following, \(\brac[X]{\cdot, \cdot}\) is the \(\LTWO\)-inner product of scalar- or vector-valued functions on a set \(X\) for its natural measure. 

\subsection{HHO scheme}

Let \(\calH\subset(0, \infty)\) be a countable set of mesh sizes with a unique cluster point at \(0\). For each \(h\in\calH\), we partition the domain \(\Omega\) into a mesh \(\Mh=(\Th, \Fh)\), for which a detailed definition can be found in \cite[Definition 1.4]{di-pietro.droniou:2020:hybrid}. The set of mesh elements \(\Th\) is a disjoint set of polytopes such that \(\ol{\Omega}=\bigcup_{T\in\Th}\ol{T}\). The set \(\Fh\) is a collection of mesh faces forming a partition of the mesh skeleton, i.e. \(\bigcup_{T\in\Th}\bdryT=\bigcup_{F\in\Fh}\ol{F}\). The boundary faces \(F\subset\partial \Omega\) are gathered in the set \(\Fhb\). The parameter \(h\) is given by \(h\defeq\max_{T\in\Th}\hT\) where, for \(X=T\in\Th\) or \(X=F\in\Fh\), \(\hX\) denotes the diameter of \(X\). We shall also collect the set of faces attached to an element \(T\in\Th\) in the set \(\Fh[T]:=\{F\in\Fh:F\subset T\}\). The (constant) unit normal to \(F\in\Fh[T]\) pointing outside \(T\) is denoted by \(\norTF\), and \(\norT:\bdryT\to\R^d\) is the piecewise constant outer unit normal defined by \((\norT)|_F=\norTF\) for all \(F\in\Fh[T]\). Throughout this work we make the following assumption on the meshes, which allows for some meshes with arbitrarily large numbers of face in each element, or faces that have an arbitrarily small diameter compared to their elements' diameters.

\begin{assumption}[Regular mesh sequence]\label{assum:star.shaped}	
	There exists a constant \(\varrho>0\) such that, for each \(h\in\calH\), each \(T\in\Th\) is connected by star-shaped sets with parameter \(\varrho\), as defined in \cite[Definition 1.41]{di-pietro.droniou:2020:hybrid}. 
\end{assumption}

From hereon, we shall denote \(f\lesssim g\) to mean \(f \le Cg\) where \(C\) is a constant depending only on \(\Omega\), \(d\) and \(\varrho\), but independent of the considered face/element, the degrees of the considered polynomial spaces, and quantities \(f,g\). We shall also write \(f\approx g\) if \(f\lesssim g\) and \(g\lesssim f\). When necessary, we make some additional dependencies of the constant \(C\) explicit.

\subsubsection{Local construction}

Let \(X=T\in\Th\) or \(X=F\in\Fh\) be a face or an element in a mesh \(\Mh\), and let \(\POLY{\ell}(X)\) be the set of \(d_X\)-variate polynomials of degree \(\le \ell\) on \(X\), where \(d_X\) is the dimension of \(X\).
The space of piecewise discontinuous polynomial functions on an element boundary is given by
\begin{equation}\label{eq:bdry.space.def}
\POLY{\ell}(\Fh[T]) \defeq \{v\in L^1(\bdryT):v|_F\in\POLY{\ell}(F)\quad\forall F\in\Fh[T]\}.
\end{equation}
The \(L^2\) orthogonal projector \(\piX{0, \ell}:L^1(X) \to \POLY{\ell}(X)\) is defined as the unique polynomial satisfying
\begin{equation}\label{eq:L2proj.def}
	\brac[X]{v - \piX{0, \ell}v, w} = 0 \qquad \forall w \in \POLY{\ell}(X).
\end{equation}

Fix two natural numbers $k,l\in\N$, $l \ge k-1$. For each element $T\in\Th$, the local space of unknowns is defined as
\[
	\UTkl\defeq \POLY{l}(T)\times\POLY{k}(\Fh[T]). 
\]
The interpolator $\ITkl:\HONE(T)\to\UTkl$ is defined for all $v\in\HONE(T)$ as
\[
	\ITkl v = (\piTzr{l}v, \piFTzr{k}v)
\]
where \(\piFTzr{k}\) is the projector onto the space \(\POLY{k}(\Fh[T])\) satisfying \(\piFTzr{k}v|_F = \piFzr{k}v\) for all \(F\in\Fh[T]\) and $v\in L^1(\partial T)$. We endow the space \(\UTkl\) with the discrete energy-like seminorm \(\norm[1, T]{{\cdot}}\) defined for all \(\ulvT = (\vT, \vFT)\in\UTkl\) via
\begin{equation}\label{eq:discrete.norm.def}
	\norm[1, T]{\ulvT}^2 \defeq \norm[T]{\nabla \vT}^2 + \hT^{-1}\norm[\bdryT]{\vFT - \vT}^2.
\end{equation}

On each element we locally reconstruct a potential from the space of unknowns via the operator $\pT{k+1}:\UTkl\to\POLY{k+1}(T)$ defined to satisfy, for all $\ulvT\in\UTkl$ and \(w\in\POLY{k+1}(T)\),
\begin{equation}\label{eq:pT.def}
	\brac[\T]{\nabla\pT{k+1}\ulvT, \nabla w} = -\brac[\T]{\vT, \Delta w} + \brac[\bdryT]{\vFT,\nabla w\cdot\norT},
\end{equation}
\begin{equation}\label{eq:pT.closure}
	\brac[\T]{\vT-\pT{k+1}\ulvT, 1} = 0. 
\end{equation}
This potential reconstruction allows us to approximate $\a(u,v)$ on each element by the bilinear form $\aT:\UTkl\times\UTkl\to\R$ defined as
\begin{equation*}
	\aT(\uluT, \ulvT) \defeq \brac[\T]{\nabla\pT{k+1}\uluT, \nabla\pT{k+1}\ulvT} + \sT(\uluT,\ulvT),
\end{equation*}
where $\sT:\UTkl\times\UTkl\to\R$ is a symmetric, positive semi-definite stabilisation such that
\begin{equation}\label{eq:norm.equivalence}
	C_{\rma}^{-1}\norm[1, T]{\ulvT}^2 \le \aT(\ulvT,\ulvT) \le C_{\rma}\norm[1, T]{\ulvT}^2
\end{equation}
and for all $\ulvT\in\UTkl$, $w\in\POLY{k+1}(T)$,
\begin{equation}\label{eq:polynomial.consistency}
	\sT(\ulvT,\ITkl w) = 0,
\end{equation}
where $C_{\rma}$ is a positive constant that possibly depends on polynomial degrees $l$, $k$, the mesh regularity $\varrho$, and dimension $d$, but is independent of the element diameter $\hT$. Equation \eqref{eq:norm.equivalence} is required to ensure that the global bilinear form describes a norm on the discrete space, and that optimal approximation rates with respect to $h$ are achieved \cite[Lemma 2.18]{di-pietro.droniou:2020:hybrid}. However, tracking the dependency of $C_{\rma}$ with respect to $l$, $k$ and obtaining condition number estimates via equation \eqref{eq:norm.equivalence} leads to sub-optimal results. As such, we assume the following extra, more precise, conditions on the bilinear form $\sT$, in which the difference operators \(\deltaT{l}:\UTkl\to\POLY{l}(T)\) and \(\deltaFT{k}:\UTkl\to\POLY{k}(\Fh[T])\) are defined as: for all \(\ulvT\in\UTkl\),
\[
	\deltaT{l} \ulvT \defeq \piTzr{l}(\pT{k+1}\ulvT - \vT) \quad\textrm{and}\quad \deltaFT{k} \ulvT \defeq \piFTzr{k}(\pT{k+1}\ulvT - \vFT).
\]

\begin{assumption}\label{assum:aT}
	For all $\ulvT\in\UTkl$ it holds that
	\begin{equation}\label{eq:aT.lower.bound}
		\norm[T]{\nabla\pT{k+1}\ulvT}^2 + \hT^{-1}\norm[\bdryT]{\deltaFT{k}\ulvT}^2 \lesssim \aT(\ulvT,\ulvT),
	\end{equation}
	and for all $\ul{v}_{T,\partial}=(0,\vFT)\in\UTkl$ it holds that
	\begin{equation}\label{eq:aT.faces.upper.bound}
		\aT(\ul{v}_{T,\partial},\ul{v}_{T,\partial}) \lesssim \norm[T]{\nabla\pT{k+1}\ul{v}_{T,\partial}}^2 + \hT^{-1}\norm[\bdryT]{\deltaFT{k}\ul{v}_{T,\partial}}^2,
	\end{equation}
	where the hidden constants in \eqref{eq:aT.lower.bound} and \eqref{eq:aT.faces.upper.bound} depend on $\varrho$ and $d$ but are independent of $l$, $k$ and $h$.
\end{assumption}

We consider throughout this work the stabilisation form defined in \cite[Example 2.8]{di-pietro.droniou:2020:hybrid} (with the scaling change $\hF^{-1}\to\hT^{-1}$)
\begin{equation}\label{eq:stab.def}
	\sT(\uluT, \ulvT) \defeq \hT^{-2}\brac[\T]{\deltaT{l}\uluT, \deltaT{l}\ulvT} + \hT^{-1}\brac[\bdryT]{\deltaFT{k}\uluT, \deltaFT{k}\ulvT}.
\end{equation}

We show in Section \ref{sec:analysis.stab} that the stabilisation \eqref{eq:stab.def} satisfies Assumption \ref{assum:aT}.

\subsubsection{Global formulation}

The global space of unknowns is defined as 
\[
	\Uhkl\defeq \Big\{\ulvh=((\vT)_{T\in\Th},(\vF)_{F\in\Fh})\,:\,\vT\in\POLY{l}(T)\quad\forall T\in\Th\,,
	\vF\in\POLY{k}(F)\quad\forall F\in\Fh \Big\}.
\]
To account for the homogeneous boundary conditions, the following subspace is also introduced:
\[
	\Uhklzr\defeq\{\ulvh \in\Uhkl:v_{F}=0\quad\forall F\in\Fhb\}.
\]
For any \(\ulvh\in\Uhkl\) we denote its restriction to an element \(T\) by \(\ulvT=(\vT,\vFT)\in\UTkl\) (where, naturally, \(\vFT\) is defined from \((\vF)_{F\in\Fh[T]}\)). We also denote by \(\vh\) the piecewise polynomial function satisfying \(\vh|_T=\vT\) for all \(T\in\Th\).

The global bilinear forms \(\ah:\Uhkl\times\Uhkl\to\mathbb{R}\) and \(\sh:\Uhkl\times\Uhkl\to\mathbb{R}\) are defined as
\[
	\ah(\uluh, \ulvh) \defeq \sum_{T\in\Th} \aT(\uluT,\ulvT)
	\quad\textrm{and}\quad
	\sh(\uluh, \ulvh) \defeq \sum_{T\in\Th} \sT(\uluT,\ulvT).
\]
We also define the discrete energy norm \(\energynorm{{\cdot}}\) on \(\Uhklzr\) as
\begin{equation}\label{eq:energy.norm.def}
	\energynorm{\ulvh}\defeq \ah(\ulvh,\ulvh)^\frac{1}{2} \qquad \forall \ulvh\in\Uhklzr.
\end{equation}	
The HHO scheme reads: find \(\uluh\in\Uhklzr\) such that
\begin{equation}\label{eq:discrete.problem}
	\ah(\uluh, \ulvh) = \mathcal L_h(\ulvh) \qquad\forall \ulvh\in\Uhklzr, 
\end{equation}
where \(\mathcal L_h:\Uhklzr\to\R\) is a linear form defined as
\begin{equation*}
	\mathcal L_h(\ulvh) \defeq \sum_{T\in\Th}\brac[T]{f,\vT}. 
\end{equation*}

Under assumptions \eqref{eq:norm.equivalence} and \eqref{eq:polynomial.consistency} on the bilinear form $\aT$, the scheme \eqref{eq:discrete.problem} satisfies the energy error estimate
\[
	\energynorm{\uluh - \Ihkl u} \le Ch^{k+1}\seminorm[\HS{k+2}(\Th)]{u},
\]
where $\Ihkl|_T=\ITkl$ for all $T\in\Th$, and $C$ is a positive constant that depends on $l$, $k$, $\varrho$, and $d$, but is independent of $h$ \cite[Theorem 2.27]{di-pietro.droniou:2020:hybrid}. An estimate of the dependency with respect to $l$ and $k$ for a diffusion scheme with a boundary based stabilisation is provided in \cite{aghili.di-pietro.ea:2017:hp-HHO}.

\subsubsection{Statically condensed system and eigenvalue estimates}\label{sec:eig.estimates}

The static condensation procedure, as outlined in \cite[Appendix B.3]{di-pietro.droniou:2020:hybrid}, allows for the elimination of the element unknowns. Selecting $\ulvh$ with one free element component $v_T$, and all other element and face components vanishing, we see that the solution \(\uluh\) to problem \eqref{eq:discrete.problem} satisfies for all \(T\in\Th\) and \(\vT\in\POLY{l}(T)\)
\[
	\aT((\uT, \uFT), (\vT, 0)) = \brac[\T]{f, \vT}.
\]
This can be alternatively written as
\[
	\aT((\uT, 0), (\vT, 0)) = \brac[\T]{f, \vT} - \aT((0, \uFT), (\vT, 0)).
\]
Noting that the bilinear form $(u_T,v_T)\in\POLY{k}(T)\times\POLY{k}(T)\mapsto \aT((u_T,0),(v_T,0))$ is coercive (due to \eqref{eq:norm.equivalence}),
we can define the polynomial \(\gT\in\POLY{l}(T)\) and the linear operator \(\calST:\POLY{k}(\Fh[T])\to\POLY{l}(T)\) via
\begin{alignat}{2}
	\aT((\gT, 0) , (\vT, 0)) \eq \brac[\T]{f, \vT} \quad&\forall\vT\in\POLY{l}(T), \label{eq:def:gT} \\
	\aT((\calST\uFT, 0) , (\vT, 0)) \eq -\aT((0, \uFT) , (\vT, 0)) \quad&\forall\vT\in\POLY{l}(T). \label{eq:ST.def}
\end{alignat}
Therefore, \(\uT\) is calculated from \(\uFT\) via the affine transformation
\begin{equation}\label{eq:uT}
	\uT = \calST\uFT + \gT.
\end{equation}
Substituting \eqref{eq:uT} into \eqref{eq:discrete.problem} and testing against \(\ulvh = (0, \vFh) = (\brac[T\in\Th]{0}, \brac[F\in\Fh]{\vF}) \in \Uhklzr\) yields
\begin{equation*}
	\sum_{T\in\Th}\aT((\calST\uFT, \uFT) , (0, \vFT)) + \sum_{T\in\Th} \aT((\gT, 0) , (0, \vFT)) = 0.
\end{equation*}
Setting
\[
	\POLY{k}_0(\Fh):=\{\uFh=(u_F)_{F\in\Fh}\,:\,u_F\in\POLY{k}(F)\quad\forall F\in\Fh\,,\quad u_F=0\mbox{ if $F\subset\partial\Omega$}\},
\]
the statically condensed problem then reads: find \(\uFh\in\POLY{k}_0(\Fh)\) such that
\begin{equation}\label{hho:statically.condensed}
	\Ah(\uFh, \vFh) = \Lh(\vFh)\quad\forall\vFh\in\POLY{k}_0(\Fh),
\end{equation}
where
\begin{align}
	\Ah(\uFh, \vFh) \defeq{}& \sum_{T\in\Th}\aT((\calST\uFT, \uFT) , (0, \vFT)),
	\label{def:Ah.sc}\\
	\Lh(\vFh) \defeq{}& \sum_{T\in\Th} -\aT((\gT, 0) , (0, \vFT)).\nonumber
\end{align}

Upon choosing bases of the spaces $\POLY{k}(F)$ for $F\in\Fhi$, \eqref{hho:statically.condensed} takes the equivalent algebraic form
\[
\mat{A}_h\bmU = \bmF
\]
where $\mat{A}_h$ is the matrix of the bilinear form $\Ah$, $\bmU$ the vector of unknowns and $\bmF$ the source term corresponding to $\Lh$.
Our main result is the following; its proof is given in Section \ref{sec:proof.eigs}.

\begin{theorem}[Eigenvalue and condition number estimates]\label{th:estimates}
For each $F\in\Fhi$, denote by $T_F^+,T_F^-$ the two elements on each side of $F$, and define the characteristic lengths $\Hmin{\Fh}$ and $\Hmax{\Fh}$ by
\[
\Hmin{\Fh}=\min_{F\in\Fh}\left(h_{T_F^+}+h_{T_F^-}\right)\,,\quad\Hmax{\Fh}^{-1}=\max_{F\in\Fh}\left(h_{T_F^+}^{-1}+h_{T_F^-}^{-1}\right).
\]
If, for each $F\in\Fhi$, the basis on $\POLY{k}(F)$ is orthonormal for the $L^2(F)$-inner product, then
the minimal eigenvalue, maximal eigenvalue and condition number of $\mat{A}_h$ satisfy
\begin{subequations}\label{est:lambda.kappa}
\begin{align}
\label{est:lambda.min}
\lambda_{\rm min}(\mat{A}_h)\gtrsim{}& \Hmin{\Fh}\,,\\
\label{est:lambda.max}
\lambda_{\rm max}(\mat{A}_h)\lesssim{}& (k+1)^2 \Hmax{\Fh}^{-1}\,,\\
\label{est:kappa}
\kappa(\mat{A}_h)\lesssim{}& (k+1)^2 \Hmax{\Fh}^{-1}\Hmin{\Fh}^{-1}.
\end{align}
\end{subequations}
\end{theorem}

\begin{remark}[Characteristic lengths]\label{rem:HminHmax}
Setting $h_{\rm min}=\min_{T\in\Th}\hT$, we have $\Hmin{\Fh}\gtrsim h_{\rm min}$ and $\Hmax{\Fh}^{-1}\lesssim h_{\rm min}^{-1}$.
Hence, \eqref{est:lambda.kappa} leads to the bounds 
\[
\lambda_{\rm min}(\mat{A}_h)\gtrsim h_{\rm min}\,,\quad\lambda_{\rm max}(\mat{A}_h)\lesssim (k+1)^2h_{\rm min}^{-1}\,,\quad\kappa(\mat{A}_h)\lesssim  (k+1)^2 h_{\rm min}^{-2}.
\]
For quasi-uniform meshes, $h_{\rm min}$ can be replaced above by $h$ in these estimates.
However, on specific meshes (especially cut meshes with small cut elements), \eqref{est:lambda.kappa} can lead to much better estimates than those purely based on $h_{\rm min}$; see Section \ref{sec:numerical}.

Note that the factor $(k+1)^2$ appearing in \eqref{est:lambda.max} and \eqref{est:kappa} is due to the dependency on the polynomial degree of
the generic discrete trace inequality \eqref{eq:discrete.trace}.
\end{remark}

\begin{remark}[Small faces]
The estimates \eqref{est:lambda.kappa} are fully independent of the maximum number of faces in each element, or on their diameter, and are therefore fully robust with respect to small faces.
\end{remark}

\section{Proofs}\label{sec:proofs}

\subsection{Estimate on the eigenvalues}\label{sec:proof.eigs}

Let us start with two preliminary estimates. The proof of the following trace inequality, under Assumption \ref{assum:star.shaped} (and therefore with hidden constants in $\lesssim$ that are not impacted by the presence of small faces in $T$), can be found in \cite[Section 3]{cangiani.dong.ea:2017:discontinuous}.

\begin{lemma}[Trace Inequality]
	For all \(v \in \HONE(T)\), 
	\begin{equation}\label{eq:continuous.trace}
		\norm[\bdryT]{v}^2 \lesssim \hT^{-1} \Big(\norm[\T]{v}^2 + \hT^2\norm[\T]{\nabla v}^2\Big).
	\end{equation}
	For $v\in \POLY{\ell}(T)$, the following discrete trace inequality also holds:
	\begin{equation}\label{eq:discrete.trace}
		\norm[\bdryT]{v}^2 \lesssim \hT^{-1}(\ell+1)(\ell+d)\norm[\T]{v}^2.
	\end{equation}
\end{lemma}

\begin{lemma}[Poincar\'{e}--Wirtinger]
	For all \(v\in H^1(T)\) the following Poincar\'{e}--Wirtinger inequality holds: 
	\begin{equation}\label{eq:poincare}
			\norm[T]{v - \piTzr{0}v} \lesssim \hT \seminorm[\HS{1}(T)]{v}.
	\end{equation}
\end{lemma}
\begin{proof}
	See \cite[Remark 1.46]{di-pietro.droniou:2020:hybrid}.
\end{proof}

\begin{lemma}[Discrete Poincar\'{e} inequality]
	For all $\ulvh\in\Uhklzr$ it holds that
	\begin{equation}\label{eq:discrete.poincare.faces}
	\sum_{T\in\Th}\hT\norm[\bdryT]{\vFT}^2 \lesssim 	\sum_{T\in\Th}\Brac{\norm[T]{\nabla\pT{k+1}\ulvT}^2 + \hT^{-1}\norm[\bdryT]{\deltaFT{k}\ulvT}^2},
	\end{equation}
	where the hidden constant depends on $d$, $\varrho$ and $\Omega$ but is independent of $l$, $k$ and $h$.
\end{lemma}

\begin{proof}
	By a triangle inequality it holds that
	\[
	\sum_{T\in\Th}\hT\norm[\bdryT]{\vFT}^2 \lesssim 	\sum_{T\in\Th}\hT\norm[\bdryT]{\piFTzr{k}\pT{k+1}\ulvT}^2 + \sum_{T\in\Th}\hT\norm[\bdryT]{\deltaFT{k}\ulvT}^2.
	\]
	The second term clearly satisfies the desired bound due to $\hT \le {\rm diam}(\Omega)^2 \hT^{-1}$. It holds by the boundedness of $\piFTzr{k}\pT{k+1}\ulvT$ and the continuous trace inequality \eqref{eq:continuous.trace} that
	\[
	\sum_{T\in\Th}\hT\norm[\bdryT]{\piFTzr{k}\pT{k+1}\ulvT}^2 \lesssim \sum_{T\in\Th}\Brac{\norm[T]{\pT{k+1}\ulvT}^2 + \hT^2\norm[T]{\nabla\pT{k+1}\ulvT}^2}.
	\]
	Thus it remains to prove that
	\[
	\norm[\Omega]{\ph{k+1}\ulvh}^2 \lesssim \sum_{T\in\Th}\Brac{\norm[T]{\nabla\pT{k+1}\ulvT}^2 + \hT^{-1}\norm[\bdryT]{\deltaFT{k}\ulvT}^2}.
	\]
	As the divergence operator \(\nabla\cdot:\HONE(\Omega)^d\to\LTWO(\Omega)\) is onto, there exists a \(\bmtau\in\HONE(\Omega)^d\) such that \(-\nabla\cdot\bmtau = \ph{k+1}\ulvh\) and \(\norm[H^1(\Omega)^d]{\bmtau} \lesssim \norm[\Omega]{\ph{k+1}\ulvh}\) \cite[Lemma 8.3]{di-pietro.droniou:2020:hybrid}. Therefore
	\begin{align*}
	\norm[\Omega]{\ph{k+1}\ulvh}^2 = -\brac[\Omega]{\ph{k+1}\ulvh,\nabla \cdot \bmtau} \eq \sum_{T\in\Th}\Brac{\brac[T]{\nabla\pT{k+1}\ulvT, \bmtau} - \brac[\bdryT]{\pT{k+1}\ulvT, \bmtau\cdot \norT}} \nl
	\eq \sum_{T\in\Th}\Brac{\brac[T]{\nabla\pT{k+1}\ulvT, \bmtau} + \brac[\bdryT]{\vFT - \pT{k+1}\ulvT, \bmtau\cdot\norT}},
	\end{align*}
	where we have invoked the homogeneous conditions on the space $\Uhklzr$ and the fact that $\bmtau\cdot\bmn_{TF}+\bmtau\cdot\bmn_{T'F}=0$ whenever $T,T'$ are the two elements on each side of an internal face $F\in\Fhi$. Thus, by the Cauchy--Schwarz inequality and continuous trace inequalities it holds that
	\begin{align*}
	\norm[\Omega]{\ph{k+1}\ulvh}^2 \les \sum_{T\in\Th}\norm[\HONE(T)]{\bmtau}\Brac{\norm[T]{\nabla\pT{k+1}\ulvT} + \hT^{-\frac12}\norm[\bdryT]{\vFT - \pT{k+1}\ulvT}} \nl
	\les \sum_{T\in\Th}\norm[\HONE(T)]{\bmtau}\Brac{\norm[T]{\nabla\pT{k+1}\ulvT} + \hT^{-\frac12}\norm[\bdryT]{\deltaFT{k}\ulvT}},
	\end{align*}
	where in the second line we have added and subtracted $\piFTzr{k}\pT{k+1}\ulvT$ to the boundary term, invoked the minimisation of $\piFTzr{k}$ and applied a continuous trace inequality. The proof follows from a discrete Cauchy--Schwarz inequality and the bound $\norm[\HONE(\Omega)^d]{\bmtau} \lesssim \norm[\Omega]{\ph{k+1}\ulvh}$.
\end{proof}

\begin{lemma}
	For all $\ul{v}_{T,\partial}=(0,\vFT)\in\UTkl$, it holds that
	\begin{equation}\label{eq:discrete.upper.bound}
			\norm[T]{\nabla\pT{k+1}\ul{v}_{T,\partial}}^2 + \hT^{-1}\norm[\bdryT]{\deltaFT{k}\ul{v}_{T,\partial}}^2 \lesssim (k+1)^2\hT^{-1}\norm[\bdryT]{\vFT}^2,
	\end{equation}
	where the hidden constant depends on $d$ and $\varrho$ but is independent of $l$, $k$ and $h$.
\end{lemma}

\begin{proof}
	By a triangle inequality, the boundedness of $\piFTzr{k}$, and the continuous trace inequality \eqref{eq:continuous.trace} it holds that
	\[
		\hT^{-1}\norm[\bdryT]{\deltaFT{k}\ul{v}_{T,\partial}}^2 \lesssim \hT^{-1}\norm[\bdryT]{\vFT}^2 + \hT^{-2}\norm[T]{\pT{k+1}\ul{v}_{T,\partial}}^2 + \norm[T]{\nabla\pT{k+1}\ul{v}_{T,\partial}}^2.
	\]
	As the element unknown is zero, it holds by \eqref{eq:pT.closure} that $\piTzr{0} \pT{k+1}\ul{v}_{T,\partial}=0$. Thus, we may apply Poincar\'{e}--Wirtinger inequality \eqref{eq:poincare} to yield $\hT^{-2}\norm[T]{\pT{k+1}\ul{v}_{T,\partial}}^2\lesssim \norm[T]{\nabla\pT{k+1}\ul{v}_{T,\partial}}^2$. Hence, it remains to be proven that
	\begin{equation}\label{eq:pT.bdry}
		\norm[T]{\nabla\pT{k+1}\ul{v}_{T,\partial}}^2 \lesssim (k+1)^2\hT^{-1}\norm[\bdryT]{\vFT}^2.
	\end{equation}
	It follows from equation \eqref{eq:pT.def} with $w=\pT{k+1}\ul{v}_{T,\partial}$ that
	\[
		\norm[T]{\nabla\pT{k+1}\ul{v}_{T,\partial}}^2 = \brac[\bdryT]{\vFT, \nabla \pT{k+1}\ul{v}_{T,\partial}\cdot\norT}.
	\]
	Applying the discrete trace inequality \eqref{eq:discrete.trace} and $(k+1)(k+d)\lesssim (k+1)^2$ yields 
	\[
		\norm[T]{\nabla\pT{k+1}\ul{v}_{T,\partial}}^2 \lesssim \hT^{-\frac12}(k+1)\norm[\bdryT]{\vFT}\norm[T]{\nabla \pT{k+1}\ul{v}_{T,\partial}}.
	\]
	Simplifying by $\norm[T]{\nabla \pT{k+1}\ul{v}_{T,\partial}}$ and squaring yields the desired result \eqref{eq:pT.bdry}.
\end{proof}

We can now prove the estimates \eqref{est:lambda.kappa} on the eigenvalues and condition number of $\mat{A}_h$.

\begin{proof}[Proof of Theorem \ref{th:estimates}]
	We note that
	\begin{align*}
		\aT((\calS_T\uFT, \uFT) , (0, \uFT)) ={}& \aT((\calS_T\uFT, \uFT) , (\calS_T\uFT, \uFT)) - \cancel{\aT((\calS_T\uFT, \uFT) , (\calS_T\uFT, 0))}
	\end{align*}
	where the cancellation follows setting \(\vT = \calS_T\uFT\) in \eqref{eq:ST.def}. By equations \eqref{eq:discrete.poincare.faces} and \eqref{eq:aT.lower.bound}, and recalling the definition \eqref{def:Ah.sc} of $\Ah$, it thus holds that
	\begin{equation}\label{eq:Ah.lower.bound}
		\sum_{T\in\Th}\hT\norm[\bdryT]{\uFT}^2 \lesssim \Ah(\uFh, \uFh).
	\end{equation}
	Consider also
	\begin{align*}
		\aT((\calS_T\uFT, \uFT) , (0, \uFT)) \eq \aT((0, \uFT) , (0, \uFT)) + \aT((\calS_T\uFT, 0) , (0, \uFT)) \nl
		\eq \aT((0, \uFT) , (0, \uFT)) - \aT((\calS_T\uFT, 0) , (\calS_T\uFT, 0)), \nl
		\lea \aT((0, \uFT) , (0, \uFT))
	\end{align*}
	where the second line follows from equation \eqref{eq:ST.def} with $v_T=\calS_T\uFT$ and the symmetry of \(\aT\), and the conclusion from the fact that $\aT$ is semi-definite positive. Therefore, by equations \eqref{eq:aT.faces.upper.bound} and \eqref{eq:discrete.upper.bound},
	\begin{equation}\label{eq:Ah.upper.bound}
		\aT((\calS_T\uFT, \uFT) , (0, \uFT)) \lesssim  (k+1)^2\hT^{-1}\norm[\bdryT]{\uFT}^2.
	\end{equation}
	Thus, combining \eqref{eq:Ah.lower.bound} and \eqref{eq:Ah.upper.bound} it holds that
	\[
		\sum_{T\in\Th}\hT\norm[\bdryT]{\uFT}^2 \lesssim \Ah(\uFh, \uFh) \lesssim (k+1)^2\sum_{T\in\Th}\hT^{-1}\norm[\bdryT]{\uFT}^2.
	\]
	Gathering by faces (and recalling that $u_{\Fh}$ vanishes on boundary faces), we obtain
	\begin{equation}\label{final.before.orthonormal}
		\sum_{F\in\Fhi}(h_{T_F^+}+h_{T_F^-})\norm[F]{u_F}^2 \lesssim \Ah(\uFh, \uFh) \lesssim (k+1)^2\sum_{F\in\Fhi}(h_{T_F^+}^{-1}+h_{T_F^-}^{-1})\norm[F]{u_F}^2.
	\end{equation}
	Having chosen orthonormal bases on the space $\POLY{k}(F)$, and recalling the definitions of $\Hmin{\Fh}$ and $\Hmax{\Fh}$, this relation reduces to
	\begin{equation}\label{final.after.orthonormal}
	\Hmin{\Fh}\bmU\cdot\bmU \lesssim \mat{A}_h\bmU\cdot\bmU\lesssim (k+1)^2\Hmax{\Fh}^{-1}\bmU\cdot\bmU.
	\end{equation}
	The estimates \eqref{est:lambda.kappa} classically follow from these bounds.
\end{proof}

\begin{remark}[Non-orthonormal polynomial bases]
The choice of orthonormal bases allows us, in the proof above, to substitute each $\norm[F]{u_F}^2$ in \eqref{final.before.orthonormal} with the Euclidean norm of the coefficients of $u_F$ on the basis of $\POLY{k}(F)$, thus leading to the global expressions $\Hmin{\Fh}\bmU\cdot\bmU$ and $\Hmax{\Fh}^{-1}\bmU\cdot\bmU$ in \eqref{final.after.orthonormal}.
If non-orthonormal polynomial bases are chosen in some $\POLY{k}(F)$, the proof shows that $\Hmin{\Fh}$ and $\Hmax{\Fh}$ have to be adjusted
the following way: for each $F$, letting $c_F,C_F$ be positive constants such that, for the chosen basis $(q^F_i)_{i\in I_F}$ of $\POLY{k}(F)$, we have
\[
c_F\sum_{i\in I_F} \lambda_i^2 \le \norm[F]{\sum_{i\in I_F}\lambda_iq^F_i}^2\le C_F\sum_{i\in I_F} \lambda_i^2\qquad\forall (\lambda_i)_{i\in I_F}\in\R,
\]
we set
\[
\Hmin{\Fh}=\min_{F\in\Fh}c_F\left(h_{T_F^+}+h_{T_F^-}\right)\,,\quad\Hmax{\Fh}^{-1}=\max_{F\in\Fh}C_F\left(h_{T_F^+}^{-1}+h_{T_F^-}^{-1}\right).
\]
It should be noted that $c_F$ and $C_F$ might depend, for some choice of polynomial bases, on the face geometry and its size. In this case, the resulting estimates on the eigenvalues and condition number may not be robust with respect to small faces in the mesh, on the contrary to those obtained using orthonormal bases (the importance, for meshes containing distorted elements, of using orthonormal bases over, say, monomial bases was already noticed in \cite[Section B.1]{di-pietro.droniou:2020:hybrid}).
\end{remark}

\subsection{Analysis of the stabilisation}\label{sec:analysis.stab}

We prove here the validity of the stabilisation term $\sT$ defined by \eqref{eq:stab.def}, and provide a brief discussion of alternate choices of stabilisation bilinear form. As the coercivity and boundedness \eqref{eq:norm.equivalence}, and polynomial consistency \eqref{eq:polynomial.consistency} are well established for the stabilisations considered here, we only wish to show that Assumption \ref{assum:aT} holds true.

\begin{lemma}
	The stabilisation bilinear form defined by \eqref{eq:stab.def} satisfies Assumption \ref{assum:aT}.
\end{lemma}

\begin{proof}
	The lower bound \eqref{eq:aT.lower.bound} follows trivially by noting that for all $\ulvT\in\UTkl$, we have
	\[
		\aT(\ulvT, \ulvT) = \norm[T]{\nabla\pT{k+1}\ulvT}^2 + \hT^{-2}\norm[T]{\deltaT{l}\ulvT}^2 +  \hT^{-1}\norm[\bdryT]{\deltaFT{k}\ulvT}^2.
	\]
	To prove the bound \eqref{eq:aT.faces.upper.bound} it remains to show that, for all $\ul{v}_{T,\partial}=(0,\vFT)\in\UTkl$,
	\[
		\hT^{-2}\norm[T]{\deltaT{l}\ul{v}_{T,\partial}}^2 \lesssim \norm[T]{\nabla\pT{k+1}\ul{v}_{T,\partial}}^2 +  \hT^{-1}\norm[\bdryT]{\deltaFT{k}\ul{v}_{T,\partial}}^2.
	\]
	We invoke the boundedness of $\piTzr{l}$ and the Poincar\'{e}--Wirtinger \eqref{eq:poincare} inequality, valid since $\piTzr{0}\pT{k+1}\ul{v}_{T,\partial}=0$ by \eqref{eq:pT.closure}, to see that
	\begin{equation*}
		\hT^{-2}\norm[T]{\deltaT{l}\ul{v}_{T,\partial}}^2 = \hT^{-2}\norm[T]{\piTzr{l}\pT{k+1}\ul{v}_{T,\partial}}^2 \le \hT^{-2}\norm[T]{\pT{k+1}\ul{v}_{T,\partial}}^2 \lesssim \norm[T]{\nabla\pT{k+1}\ul{v}_{T,\partial}}^2,
	\end{equation*}
	thus, completing the proof.
\end{proof}

\subsubsection{Alternate choices for the stabilisation bilinear form}

We briefly comment here on a variety of different choices for the stabilisation term $\sT$. For the choice of element polynomial degree $l=k-1$, the stabilisation bilinear form $\sTkminus:\UT{k-1,k}\times\UT{k-1,k}\to\R$ defined for all $\ulvT,\ulwT\in\UT{k-1,k}$ via
\[
	\sTkminus(\ulvT, \ulwT) \defeq \hT^{-1}\brac[\bdryT]{\deltaFT{k}\ulvT, \deltaFT{k}\ulwT}
\]
satisfies the requirements \eqref{eq:norm.equivalence} and \eqref{eq:polynomial.consistency} \cite[Section 4.3]{droniou.yemm:2021:robust}. Moreover, it is clear that $\sTkminus$ satisfies Assumption \ref{assum:aT} with hidden constant in \eqref{eq:aT.lower.bound} and \eqref{eq:aT.faces.upper.bound} equal to $1$. We emphasise, however, that  when $l>k-1$ the coercivity \eqref{eq:norm.equivalence} of $\ah$ fails for this choice of stabilisation, and that the discrete problem \eqref{eq:discrete.problem} is ill posed.

Another choice of stabilisation with only boundary terms is the ``original HHO stabilisation'' $\sTbdry:\UT{l,k}\times\UT{l,k}\to\R$ defined for all $\ulvT,\ulwT\in\UT{l,k}$ via
\begin{equation}\label{eq:bdry.stab.def}
	\sTbdry(\ulvT, \ulwT) \defeq \hT^{-1}\brac[\bdryT]{(\deltaFT{k} - \deltaT{l})\ulvT, (\deltaFT{k} - \deltaT{l})\ulwT}.
\end{equation}
It satisfies the coercivity and boundedness requirements \eqref{eq:norm.equivalence} for all $l\le k+1$ \cite[Proposition 2.13]{di-pietro.droniou:2020:hybrid}. This choice of  stabilisation also satisfies the upper bound \eqref{eq:aT.faces.upper.bound} in Assumption \ref{assum:aT}, however, we have been yet unable to prove the lower bound \eqref{eq:aT.lower.bound} with a constant that does not depend on $k$.

\begin{remark}[HDG stabilisation]
In the case $l=k+1$, the following \ac{hdg}-inspired stabilisation can also be considered (see \cite[Section 5.1.6]{di-pietro.droniou:2020:hybrid} and \cite{Cockburn.Di-Pietro.ea:16}):
\[
\mathrm{s}_T^{\textsc{hdg}}(\ulvT, \ulwT) = \hT^{-1}\brac[\bdryT]{\piFTzr{k}(\vFT-\vT), \piFTzr{k}(\wFT-\wT)}.
\]
As for $\sTbdry$ above, we can prove a uniform-in-$k$ upper bound \eqref{eq:aT.faces.upper.bound} for $\mathrm{s}_T^{\textsc{hdg}}$, but the lower bounds we could establish depend on $k$.
\end{remark}

The gradient-based stabilisation $\sTgrad:\UT{l,k}\times\UT{l,k}\to\R$ introduced in \cite[Section 4]{droniou.yemm:2021:robust} is defined for all $\ulvT,\ulwT\in\UT{l,k}$ via
\begin{equation}\label{eq:grad.stab.def}
	\sTgrad(\ulvT, \ulwT) \defeq \brac[T]{\nabla \deltaT{l}\ulvT, \nabla \deltaT{l}\ulwT} + \hT^{-1}\brac[\bdryT]{\deltaFT{k}\ulvT, \deltaFT{k}\ulwT}.
\end{equation}
The gradient-based stabilisation satisfies coercivity, boundedness, and polynomial consistency for all $l\ge k-1$. Moreover, it is clear that $\sTgrad$ satisfies equation \eqref{eq:aT.lower.bound} in Assumption \ref{assum:aT}. For $l\ge k+1$, the upper bound \eqref{eq:aT.faces.upper.bound} also follows trivially. However, for $l=k-1,k$ we have been unable to prove that this choice of stabilisation satisfies \eqref{eq:aT.faces.upper.bound} without an extra dependency on $k$.

Despite these shortcomings in the analysis, numerical tests suggest that $\sTbdry$ and $\sTgrad$ satisfy the eigenvalue estimates stated in Theorem \ref{th:estimates}. This is illustrated in Figure \ref{fig:k.test}. Moreover, these choices of stabilisation might be preferable as the error induced when measured in certain norms can be significantly lower than for the choice \eqref{eq:stab.def} \cite[Figures 2, 3]{droniou.yemm:2021:robust}.

\section{HHO on cut meshes}\label{sec:unfitted}
 
As discussed in the introduction, the generation of unstructured body-fitted meshes of geometrically complex regions -- such as those with curved boundaries and high curvatures -- can present great difficulties. Unfitted finite element methods avoid this issue because they are defined on a simple (e.g., Cartesian or octree) background mesh covering the domain of interest. The elements in touch with interface boundaries can be locally cut to produce polytopal elements on the physical domain boundaries \cite{2110.01378}. These cuts can produce narrow, anisotropic `sliver-cut' elements, as well as small but round `small-cut' elements.  

The design of a variant of the HHO method on cut meshes, with potentially curved elements, is presented and analysed in \cite{burman.cicuttin.ea:2021:unfitted} for elliptic interface problems. The unfitted \ac{hho} method therein makes use of Nitsche's method for the local reconstruction operator. Instead, we consider a standard \ac{hho} method on cut meshes. In particular, we define a simple structured background mesh $\mathcal{T}_{h}^{\mathrm{bg}}$ and extract the submesh of active elements $\mathcal{T}_{h}^{\mathrm{act}}$. The active mesh is split into interior elements $\mathcal{T}_{h}^{\mathrm{in}}$ and cut elements $\mathcal{T}_{h}^{\mathrm{cut}}$. 

Based on the condition number bounds in Theorem \ref{th:estimates}, we know that the conditioning of the system matrix can be severely affected by the presence of small-cut and sliver-cut elements. To attain condition number bounds on cut meshes that are independent of the cut location, sliver-cut and small-cut elements in $\mathcal{T}_{h}^{\mathrm{cut}}$ are aggregated to their neighbours to form an isotropic, quasi-uniform mesh. In particular, we iterate over elements $T \in \mathcal{T}_{h}^{\mathrm{cut}}$ and merge $T$ with its neighbour sharing the longest edge (or face) if
\[
	\frac{|T|_d}{|\bdryT|_{d-1}} < \epsilon_1 \hT \qquad\textrm{or}\qquad \hT < \epsilon_2 \hmax.
\] 
The algorithm is re-run until no ill-posed elements are found. The convergence of this algorithm is assured, since any ill-posed cell is at finite distance to a well-posed cell. The size of the aggregates is bounded by the maximum of such distance for all ill-cells, which depends on the scale of the geometrical features (see \cite[Lemma 2.2]{badia.verdugo.ea:2018:aggregated}). We take $\epsilon_1 = 0.05$ and $\epsilon_2 = 0.3$ in the numerical experiments section. After this aggregation step, we end up with a new mesh $\mathcal{T}_{h}^{\mathrm{ag}}$.  Let us note that arbitrarily small faces can still be present in $\mathcal{T}_{h}^{\mathrm{ag}}$. The following corollary is a direct consequence of Theorem~\ref{th:estimates} and the aggregation algorithm.

\begin{corollary}[Eigenvalues and condition numbers on cut meshes] 
Let $\mathcal{T}_{h}^{\mathrm{bg}}$ be a background mesh covering $\Omega$ with characteristic mesh size $h$ and $\mathcal{T}_{h}^{\mathrm{ag}}$ the corresponding aggregated mesh obtained using the algorithm described above. Let $\mat{A}_h$ be the linear system matrix
corresponding to the \ac{hho} discretisation (\ref{hho:statically.condensed}) for $\mathcal{T}_{h}^{\mathrm{ag}}$. Under the assumptions in Theorem~\ref{th:estimates}, it holds:
\[
\lambda_{\rm min}(\mat{A}_h)\gtrsim h, \qquad 
\lambda_{\rm max}(\mat{A}_h)\lesssim (k+1)^2 h^{-1}, \qquad 
\kappa(\mat{A}_h)\lesssim (k+1)^2 h^{-2},
\]
where the constants are independent of the cut location but depend on the choice of $\epsilon_1$ and $\epsilon_2$.  
\end{corollary}

We note that the ill-conditioning of systems arising in unfitted $C^0$-Lagrangian \acp{fe} can be solved by aggregating ill-conditioned elements into their neighbours \cite{badia.verdugo.ea:2018:aggregated}. However, the strategy we consider here is simpler because there is no need to eliminate ill-posed nodes via constraints in each aggregate.

\section{Numerical Results}\label{sec:numerical}
We provide here a numerical study of the condition number to illustrate the results derived in previous sections. The linear system \eqref{hho:statically.condensed} is assembled using the \texttt{HArDCore} open source C++ library \cite{hhocode}. We compute the condition number using the \texttt{SymEigsSolver} solver found in the \texttt{Spectra} library, with documentation available at \url{https://spectralib.org/doc/index.html}.
All numerical tests in this section are performed using element degree $l=k$, and $L^2$- orthonormalised basis functions. The orthonormalisation process is achieved using a classical Gram-Schmidt algorithm.

\subsection{Coarsened meshes}

In order to capture intricate geometric details in a given domain it is sometimes sensible to start with a regular, fine mesh of small element diameter, and agglomerate elements together in order to save computation time. These coarsened meshes are (relatively) isotropic and quasi-uniform, however can have many faces per element and arbitrarily small face diameters. Thus, Theorem \ref{th:estimates} predicts the maximum and minimum eigenvalues to scale as $\lambda_{\rm{min}}(\mat{A}_h)\approx h$ and $\lambda_{\rm{max}}(\mat{A}_h)\approx h^{-1}$ respectively, independently of the number and size of faces in each element. We consider the unit box $\Omega=(0,1)^2\subset\R^2$, and a fine triangular mesh of $\Omega$. We then design  successive coarsenings of these meshes and observe how the condition number evolves. Such meshes are plotted in Figure \ref{fig:coarse.mesh} with the data of the mesh sequence presented in Table \ref{table:coarse.mesh}.

\begin{figure}[H]
	\centering
	\includegraphics[width=3\textwidth/10]{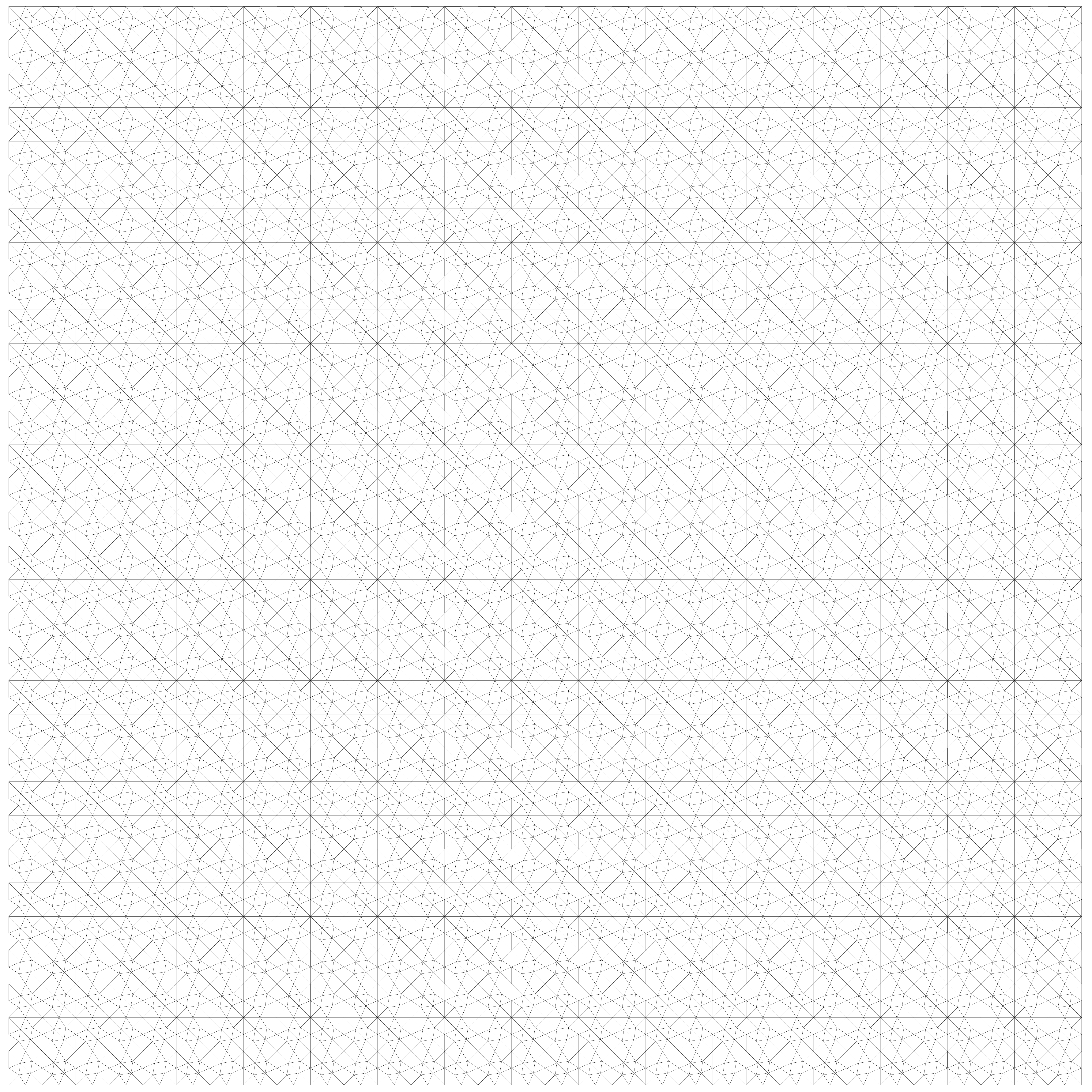} 
	\includegraphics[width=3\textwidth/10]{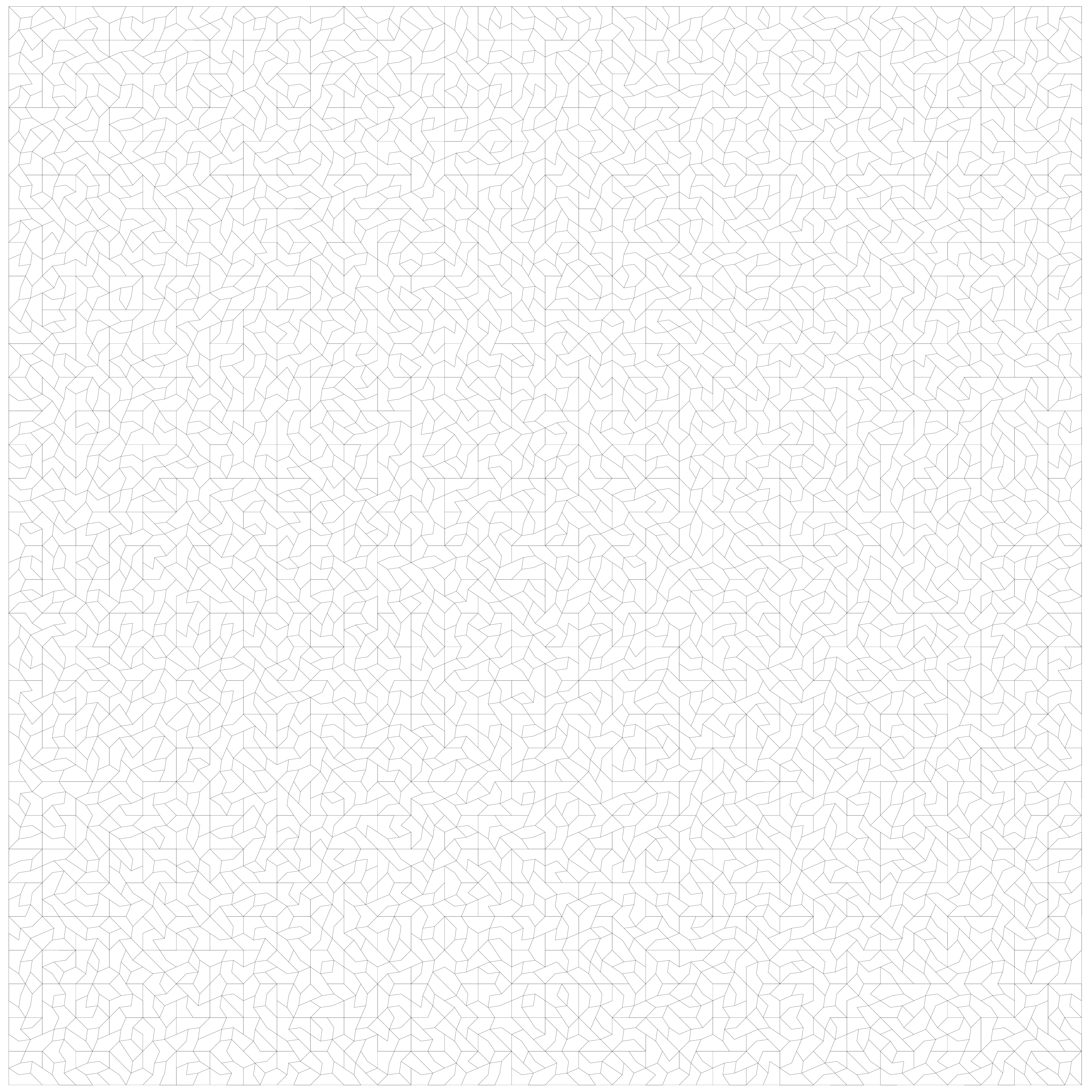} 
	\includegraphics[width=3\textwidth/10]{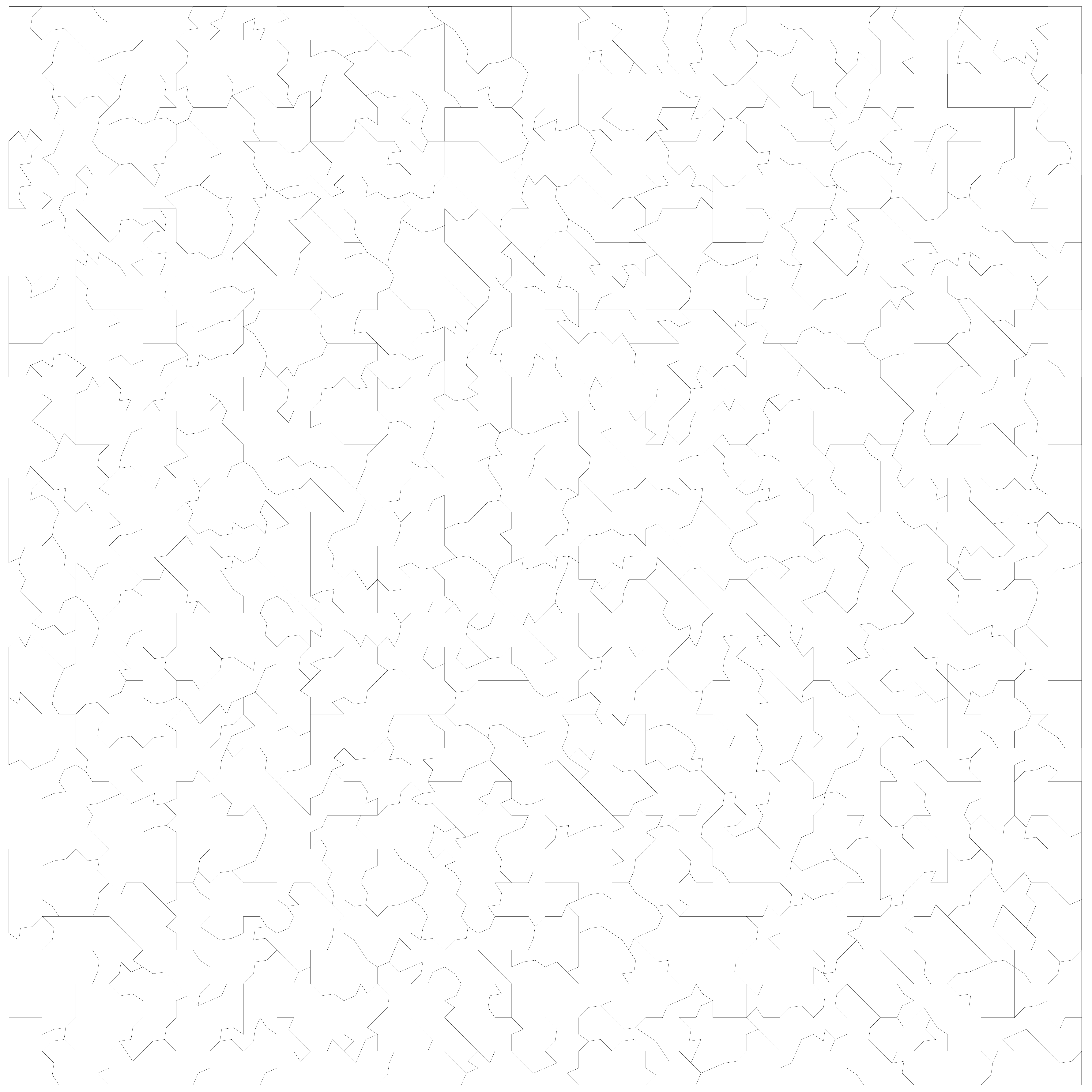} 
	\caption{Coarsened meshes}
	\label{fig:coarse.mesh}
\end{figure}

\begin{table}[H]
	\centering
	\pgfplotstableread{data/coarse_mesh_test_k0.dat}\loadedtable
	\pgfplotstabletypeset
	[
	columns={hMin,hMax,NbCells,NbInternalEdges}, 
	columns/hMin/.style={column name=\(\hmin\)},
	columns/hMax/.style={column name=\(\hmax\)},
	columns/NbCells/.style={column name=Nb. Elements},
	columns/NbInternalEdges/.style={column name=Nb. Internal Edges},
	every head row/.style={before row=\toprule,after row=\midrule},
	every last row/.style={after row=\bottomrule} 
	]\loadedtable
	\caption{Coarsened meshes}
	\label{table:coarse.mesh}
\end{table}

The condition number and eigenvalues on each mesh are plotted in Figure \ref{fig:coarse.mesh.test}. As the mesh is coarsened the condition number appears to decay slightly slower than $h^{-2}$. This is easily explainable due to the successive meshes becoming less `round', and thus the mesh regularity parameter increasing slightly with each level of coarsening.

\begin{figure}[H]
\centering
\ref{coarse_mesh_test}
\vspace{0.5cm}\\
\subcaptionbox{$\kappa(\mat{A}_h)$ vs $h^{-1}$}
{
	\begin{tikzpicture}[scale=0.57]
	\begin{loglogaxis}[legend columns=3, legend to name=coarse_mesh_test, xlabel={$h^{-1}$}, ylabel={$\kappa(\mat{A}_h)$}]
	\addplot table[x=hMaxInv, y=Condition] {data/coarse_mesh_test_vol_stab_k0.dat};
	\addlegendentry{\(k=0\)};
	\addplot table[x=hMaxInv,y=Condition] {data/coarse_mesh_test_vol_stab_k1.dat};
	\addlegendentry{\(k=1\)};
	\addplot table[x=hMaxInv,y=Condition] {data/coarse_mesh_test_vol_stab_k2.dat};
	\addlegendentry{\(k=2\)};
	\logLogSlopeTriangle{0.9}{0.4}{0.1}{2}{black};
	\end{loglogaxis}
	\end{tikzpicture}
}
\subcaptionbox{$\lambda_{\rm{min}}(\mat{A}_h)$ vs $h$}
{
	\begin{tikzpicture}[scale=0.57]
	\begin{loglogaxis}[xlabel={$h$}, ylabel={$\lambda_{\rm{min}}(\mat{A}_h)$}]
	\addplot table[x=hMax,y=MinEig] {data/coarse_mesh_test_vol_stab_k2.dat};
	\addplot table[x=hMax,y=MinEig] {data/coarse_mesh_test_vol_stab_k1.dat};
	\addplot table[x=hMax,y=MinEig] {data/coarse_mesh_test_vol_stab_k2.dat};
	\logLogSlopeTriangle{0.9}{0.4}{0.1}{1}{black};
	\end{loglogaxis}
	\end{tikzpicture}
}
\subcaptionbox{$\lambda_{\rm{max}}(\mat{A}_h)$ vs $h^{-1}$}
{
	\begin{tikzpicture}[scale=0.57]
	\begin{loglogaxis}[xlabel={$h^{-1}$}, ylabel={$\lambda_{\rm{max}}(\mat{A}_h)$}]
	\addplot table[x=hMaxInv,y=MaxEig] {data/coarse_mesh_test_vol_stab_k0.dat};
	\addplot table[x=hMaxInv,y=MaxEig] {data/coarse_mesh_test_vol_stab_k1.dat};
	\addplot table[x=hMaxInv,y=MaxEig] {data/coarse_mesh_test_vol_stab_k2.dat};
	\logLogSlopeTriangle{0.9}{0.4}{0.1}{1}{black};
	\end{loglogaxis}
	\end{tikzpicture}
}
\caption{Coarse square meshes}
\label{fig:coarse.mesh.test}
\end{figure}
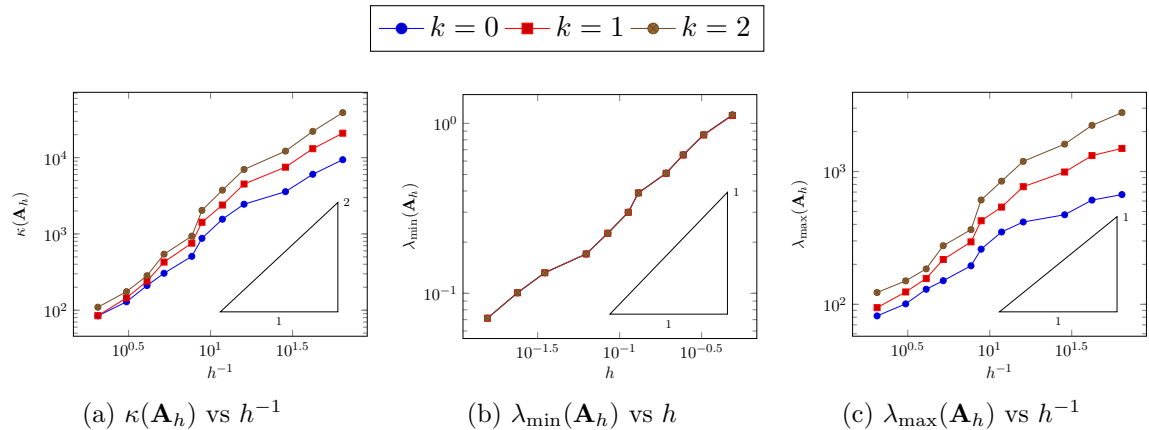

In Figure \ref{fig:k.test} we fix the mesh (the mesh with $\hmin=0.11$ in Table \ref{table:coarse.mesh}) and vary the polynomial degree $k$. We test with the stabilisation defined by \eqref{eq:stab.def} as well as the gradient-based \eqref{eq:grad.stab.def} and boundary \eqref{eq:bdry.stab.def} stabilisations. It is apparent from Figure \ref{fig:k.test} that all three stabilisations result in a system matrix with eigenvalue estimates determined by Theorem \ref{th:estimates}.

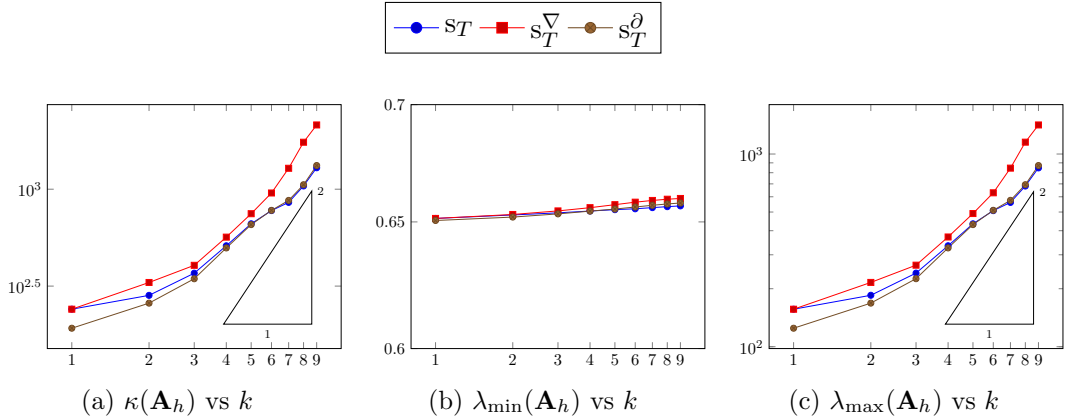
\begin{figure}[H]
\centering
\ref{ktest}
\vspace{0.5cm}\\
	\subcaptionbox{$\kappa(\mat{A}_h)$ vs $k$}
	{
		\begin{tikzpicture}[scale=0.57]
		\begin{loglogaxis}
		[ 
		legend columns=3, 
		legend to name=ktest,
		xtick={1,2,3,4,5,6,7,8,9},
		xticklabels={1,2,3,4,5,6,7,8,9}
		]
		\addplot table[x=EdgeDegree, y=Condition] {data/k_test_coarse_mesh_vol_stab.dat};
		\addlegendentry{$\sT$};
		\addplot table[x=EdgeDegree, y=Condition] {data/k_test_mesh1_5_coarse_8.dat};
		\addlegendentry{$\sTgrad$};
		\addplot table[x=EdgeDegree, y=Condition] {data/k_test_coarse_mesh_bdry_stab.dat};
		\addlegendentry{$\sTbdry$};
		\logLogSlopeTriangle{0.9}{0.3}{0.1}{2}{black};
		\end{loglogaxis}
		\end{tikzpicture}
	}
	\subcaptionbox{$\lambda_{\rm{min}}(\mat{A}_h)$ vs $k$}
	{
		\begin{tikzpicture}[scale=0.57]
		\begin{loglogaxis}
		[ 
		xtick={1,2,3,4,5,6,7,8,9},
		xticklabels={1,2,3,4,5,6,7,8,9},
		ymin=0.6,
		ymax=0.7,
		ytick={0.6,0.65,0.7},
		yticklabels={0.6,0.65,0.7}
		]
		\addplot table[x=EdgeDegree, y=MinEig] {data/k_test_coarse_mesh_vol_stab.dat};
		\addplot table[x=EdgeDegree, y=MinEig] {data/k_test_mesh1_5_coarse_8.dat};
		\addplot table[x=EdgeDegree, y=MinEig] {data/k_test_coarse_mesh_bdry_stab.dat};
		\end{loglogaxis}
		\end{tikzpicture}
	}
	\subcaptionbox{$\lambda_{\rm{max}}(\mat{A}_h)$ vs $k$}
	{
		\begin{tikzpicture}[scale=0.57]
		\begin{loglogaxis}
		[ 
		xtick={1,2,3,4,5,6,7,8,9},
		xticklabels={1,2,3,4,5,6,7,8,9}
		]
		\addplot table[x=EdgeDegree, y=MaxEig] {data/k_test_coarse_mesh_vol_stab.dat};
		\addplot table[x=EdgeDegree, y=MaxEig] {data/k_test_mesh1_5_coarse_8.dat};
		\addplot table[x=EdgeDegree, y=MaxEig] {data/k_test_coarse_mesh_bdry_stab.dat};
		\logLogSlopeTriangle{0.9}{0.3}{0.1}{2}{black};
		\end{loglogaxis}
		\end{tikzpicture}
	}
\caption{Condition number vs polynomial degree}
\label{fig:k.test}
\end{figure}

\subsection{Cut meshes}

In this section, we apply the \ac{hho} method to cut meshes using the aggregation strategy proposed above. The computation of the cut meshes and the boundary-element intersections has been carried out using the $\mathtt{Gridap}$ open-source Julia library \cite{Badia2020b} version 0.16.3 and its extension package for unfitted methods GridapEmbedded.jl~\cite{GridapEmbedded-jl} version 0.7 (see \cite{2110.01378} for more details). We consider first-order boundary representations -- that is, we consider a piecewise linear approximation of the curved boundary. 

\subsubsection{Test A}

Take the circular domain \(\Omega = \{(x, y)\in\R^2:x^2 + y^2 < 1\} \subset\R^2\) and consider three similar cut meshes of $\Omega$, with a parameter $\epsilon$ controlling the diameter of the smallest cut elements ($\epsilon < \hT$ for all $T\in\Th$). The mesh data is given in Table \ref{table:mesh.epsilon.cut}, and we plot values of the condition number and eigenvalues versus $\epsilon$ in Figure \ref{fig:small.cut.cells}. It is clear that both the maximum eigenvalue, and the condition number, become unbounded as $\epsilon\to 0$. The minimum eigenvalue, however, stays approximately constant. This is consistent with the theory as each face is connected to at least one element with diameter proportional to $\hmax$, thus we expect $\lambda_{\rm{min}}(\mat{A}_h)\sim \hmax = {\rm const}$.

\begin{table}[H]
	\centering
	\pgfplotstableread{data/small_cut_cells_test_k0.dat}\loadedtable
	\pgfplotstabletypeset
	[
	columns={Epsilon,hMin,hMax,NbCells,NbInternalEdges}, 
	columns/Epsilon/.style={column name=\(\epsilon\)}, 
	columns/hMin/.style={column name=\(\hmin\)},
	columns/hMax/.style={column name=\(\hmax\)},
	columns/NbCells/.style={column name=Nb. Elements},
	columns/NbInternalEdges/.style={column name=Nb. Internal Edges},
	every head row/.style={before row=\toprule,after row=\midrule},
	every last row/.style={after row=\bottomrule} 
	]\loadedtable
	\caption{Parameters of the circular meshes with varying values of $\epsilon$}
	\label{table:mesh.epsilon.cut}
\end{table}

\begin{figure}[H]
\centering
\ref{small_cut_cells}
\vspace{0.5cm}\\
\subcaptionbox{$\kappa(\mat{A}_h)$ vs $\epsilon$}
{
	\begin{tikzpicture}[scale=0.57]
	\begin{loglogaxis}[legend columns=3, legend to name=small_cut_cells, xlabel={$\epsilon$}, ylabel={$\kappa(\mat{A}_h)$}]
	\addplot table[x=Epsilon,y=Condition] {data/small_cut_cells_test_vol_stab_k0.dat};
	\addlegendentry{\(k=0\)};
	\addplot table[x=Epsilon,y=Condition] {data/small_cut_cells_test_vol_stab_k1.dat};
	\addlegendentry{\(k=1\)};
	\addplot table[x=Epsilon,y=Condition] {data/small_cut_cells_test_vol_stab_k2.dat};
	\addlegendentry{\(k=2\)};
	\reverseLogLogSlopeTriangle{0.5}{0.4}{0.1}{1}{black};
	\end{loglogaxis}
	\end{tikzpicture}
}
\subcaptionbox{$\lambda_{\rm{min}}(\mat{A}_h)$ vs $\epsilon$}
{
	\begin{tikzpicture}[scale=0.57]
	\begin{loglogaxis}[
	  legend columns=3, 
	  xlabel={$\epsilon$}, 
	  ylabel={$\lambda_{\rm{min}}(\mat{A}_h)$}, 
	  ymin=0.5, 
	  ymax=0.6, 
	  ytick={0.5,0.55,0.6}, 
	  yticklabels={0.5,0.55,0.6}
	  ]
	\addplot table[x=Epsilon,y=MinEig] {data/small_cut_cells_test_vol_stab_k0.dat};
	\addplot table[x=Epsilon,y=MinEig] {data/small_cut_cells_test_vol_stab_k1.dat};
	\addplot table[x=Epsilon,y=MinEig] {data/small_cut_cells_test_vol_stab_k2.dat};
	\end{loglogaxis}
	\end{tikzpicture}
}
\subcaptionbox{$\lambda_{\rm{max}}(\mat{A}_h)$ vs $\epsilon$}
{
	\begin{tikzpicture}[scale=0.57]
	\begin{loglogaxis}[legend columns=3, xlabel={$\epsilon$}, ylabel={$\lambda_{\rm{max}}(\mat{A}_h)$}]
	\addplot table[x=Epsilon,y=MaxEig] {data/small_cut_cells_test_vol_stab_k0.dat};
	\addplot table[x=Epsilon,y=MaxEig] {data/small_cut_cells_test_vol_stab_k1.dat};
	\addplot table[x=Epsilon,y=MaxEig] {data/small_cut_cells_test_vol_stab_k2.dat};
	\reverseLogLogSlopeTriangle{0.5}{0.4}{0.1}{1}{black};
	\end{loglogaxis}
	\end{tikzpicture}
}
\caption{Cut meshes with small-cut elements}
\label{fig:small.cut.cells}
\end{figure}
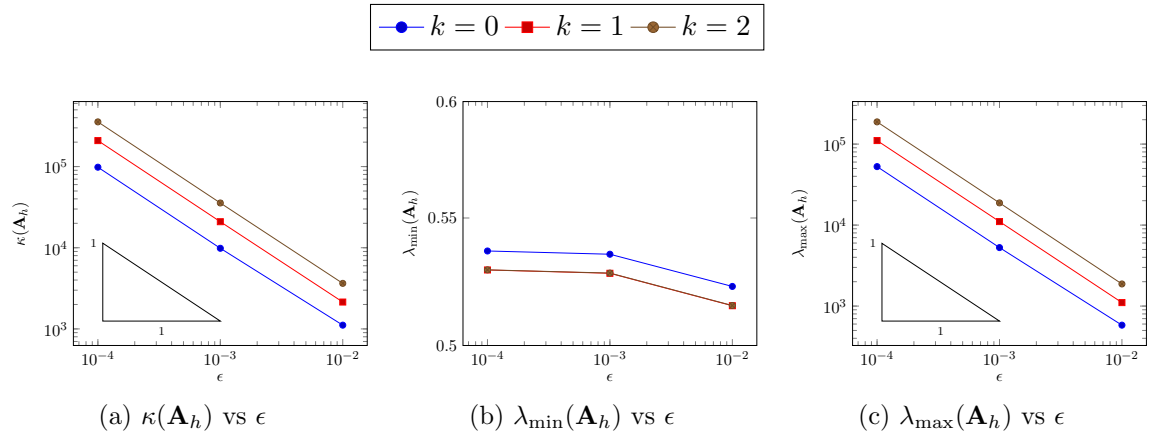

To avoid unbounded condition numbers on cut meshes, sliver and small-cut elements are aggregated as explained above. 
A portion of the resulting aggregated mesh $\mathcal{T}_{h}^{\mathrm{ag}}$ of $\Omega$ is plotted in Figure \ref{fig:cut.mesh.aggr.zoom} showing the aggregation of sliver-cut and small-cut elements. We note the existence of arbitrarily small faces after the aggregation of small-cut elements.

\begin{figure}[H]
\centering
\subcaptionbox{No aggregation}{\includegraphics[width=0.3\textwidth]{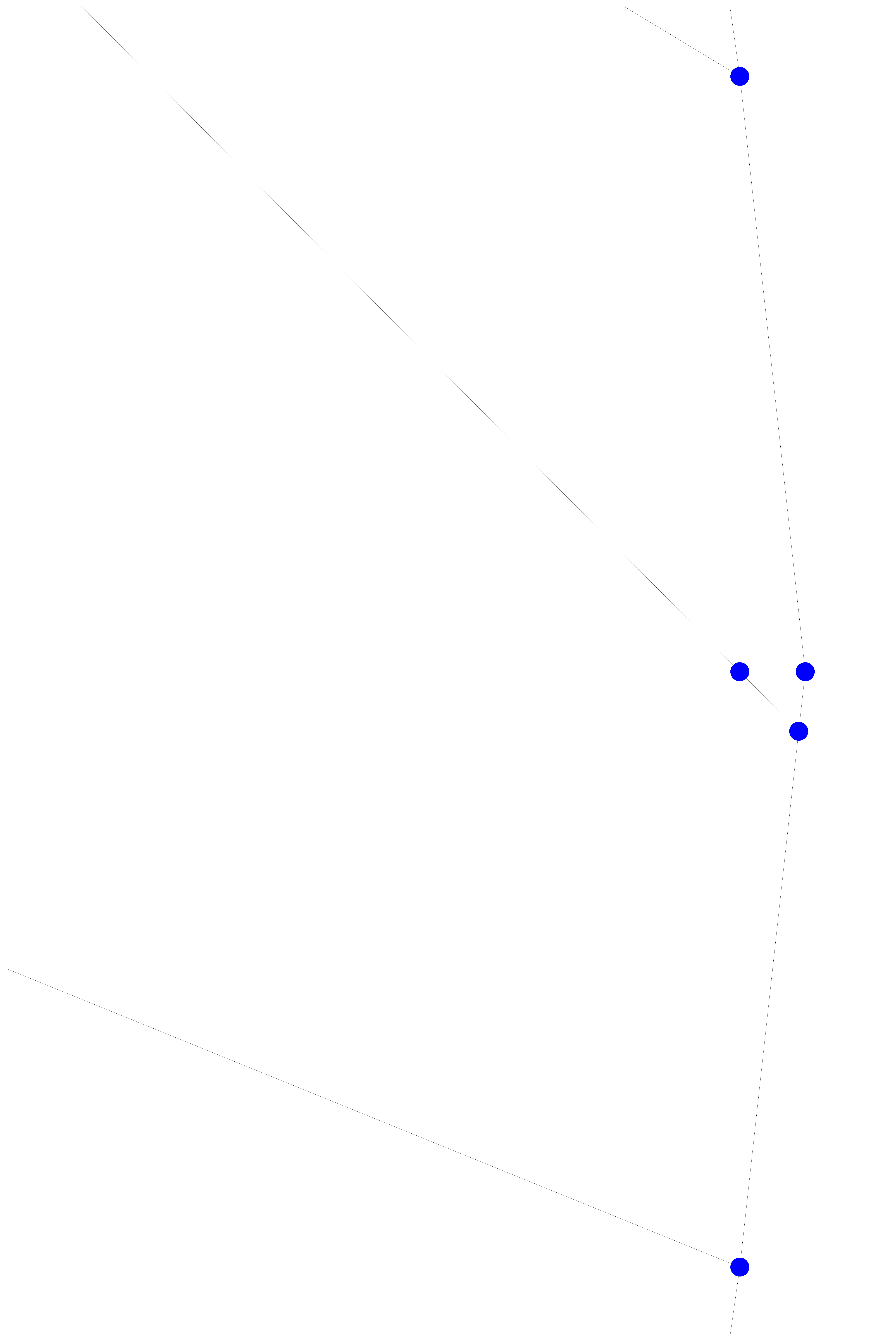}}$\ $
\subcaptionbox{Sliver-cut elements aggregated}{\includegraphics[width=0.3\textwidth]{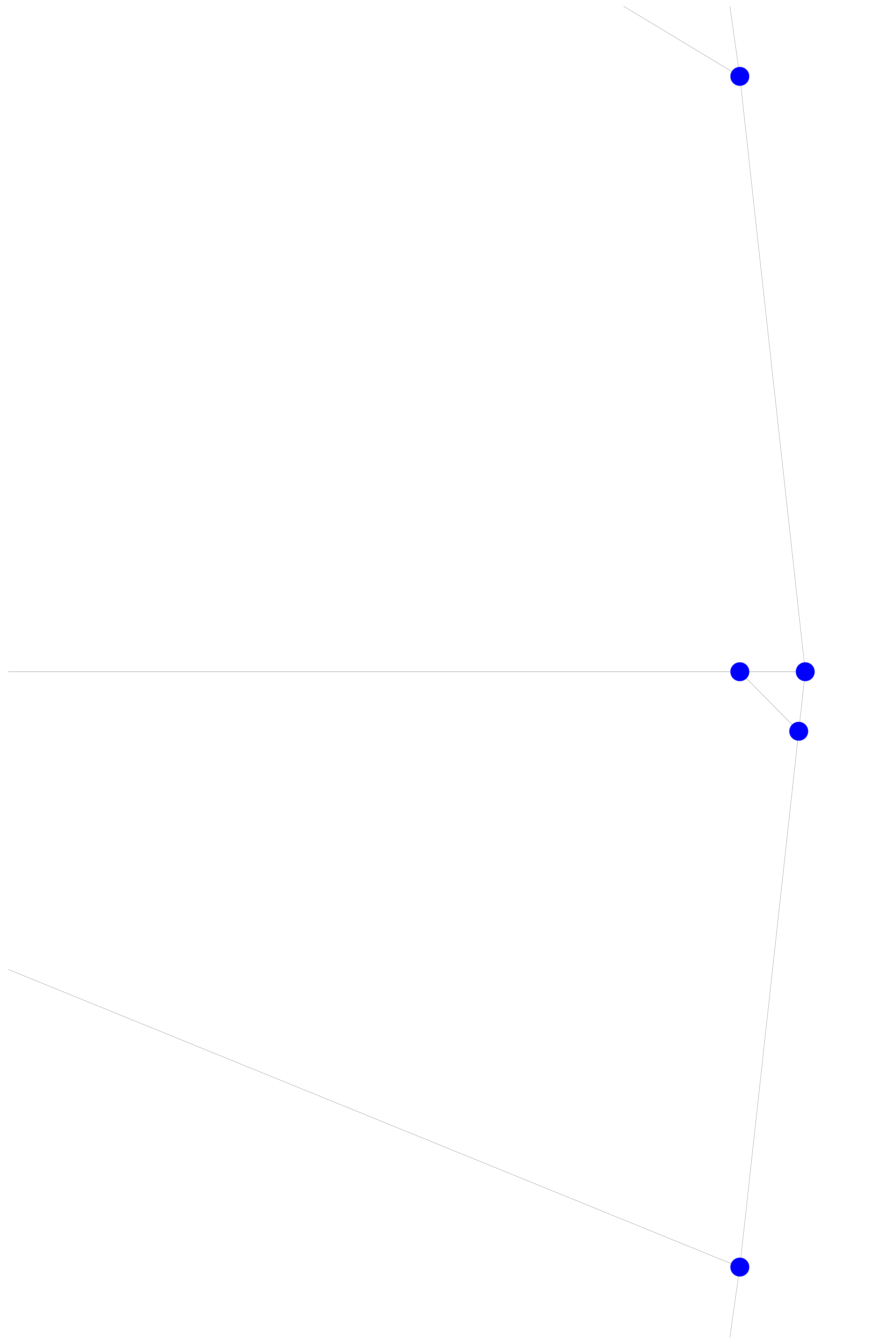}}$\ $
\subcaptionbox{Sliver-cut and small-cut elements aggregated}{\includegraphics[width=0.3\textwidth]{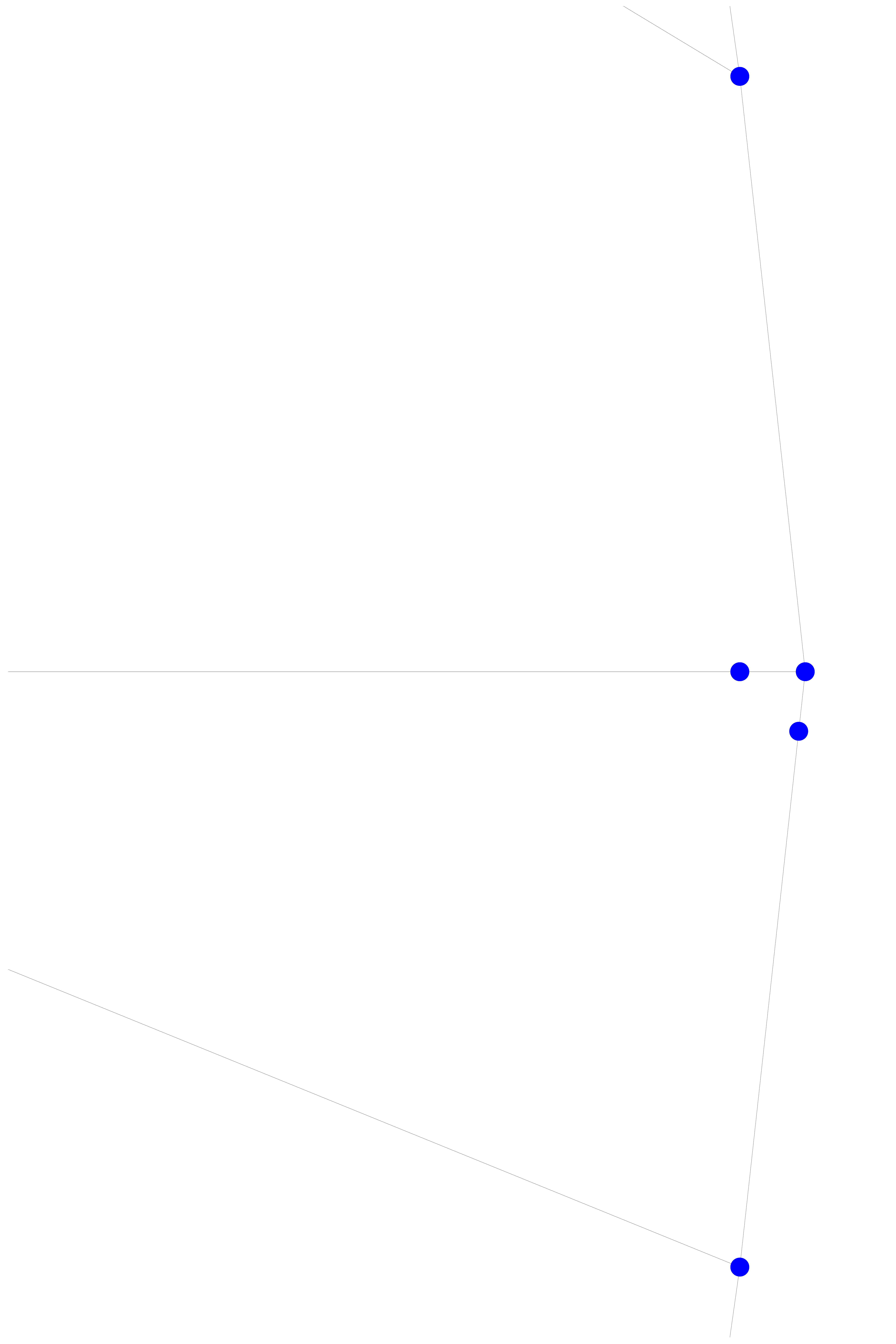}}
\caption{Aggregation of cut meshes (local)}
\label{fig:cut.mesh.aggr.zoom}
\end{figure}

\begin{figure}[H]
	\centering
	\subcaptionbox{No aggregation}{\includegraphics[width=0.495\textwidth]{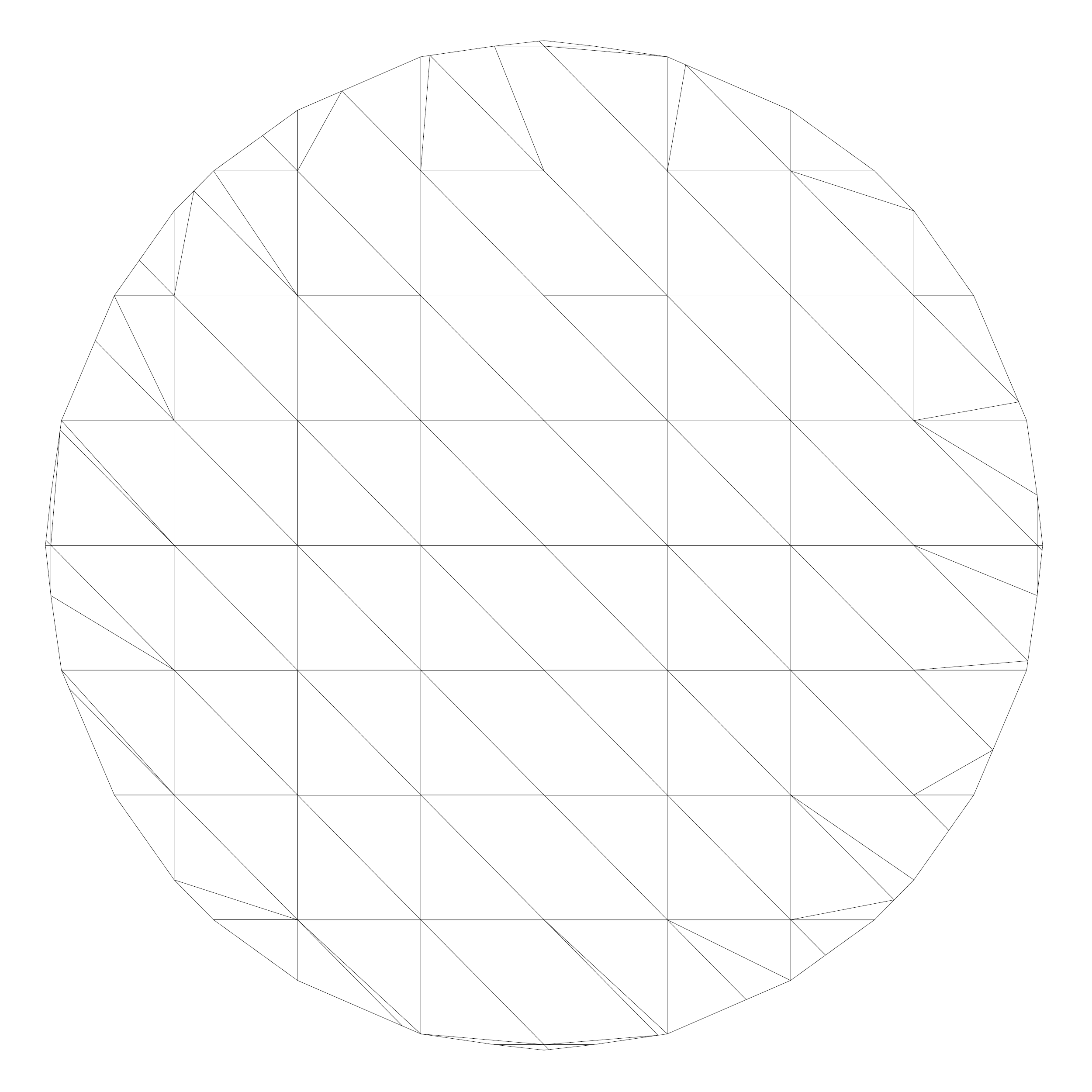}}
	\subcaptionbox{Sliver-cut and small-cut elements aggregated}{\includegraphics[width=0.495\textwidth]{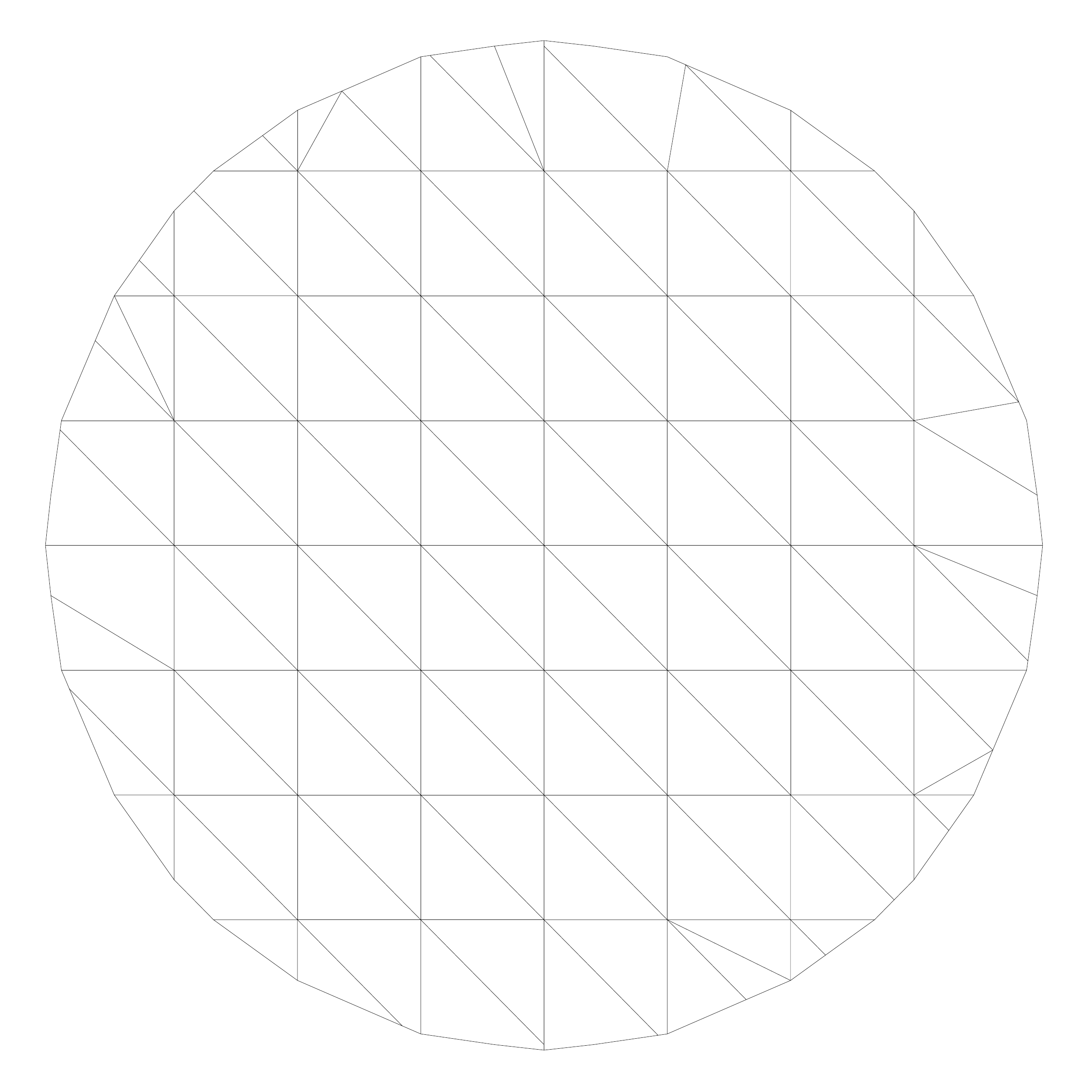}}
	\caption{Aggregation of cut meshes (global)}
	\label{fig:cut.mesh.aggr.mesh1.global}
\end{figure}

For each mesh in Table \ref{table:mesh.epsilon.cut} we consider a corresponding aggregated mesh, and in Figure \ref{fig:aggr.cells} we test the condition number and eigenvalues of the system matrix for various polynomial degrees $k$. It is clear that after aggregation the minimal eigenvalue, maximal eigenvalue and condition number are independent of $\epsilon$. 

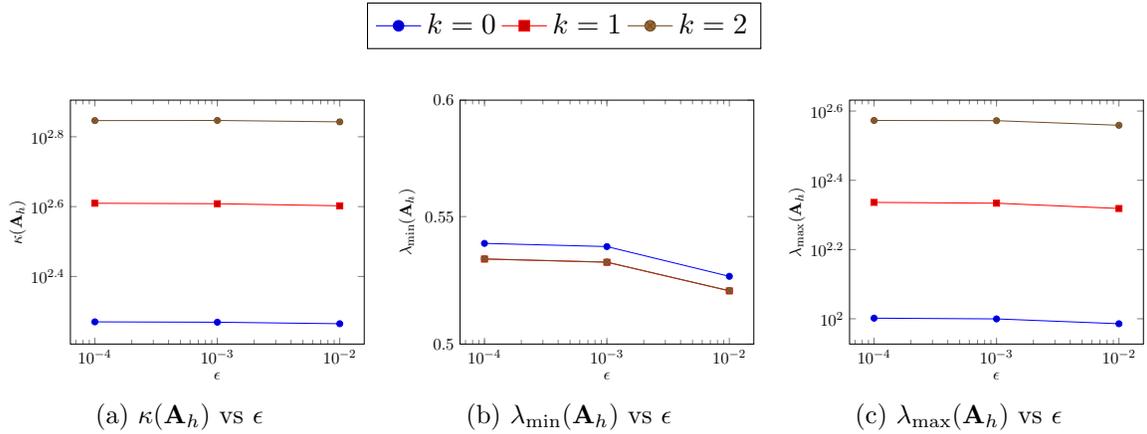
\begin{figure}[H]
\centering
\ref{aggr_cut_cells}
\vspace{0.5cm}\\
\subcaptionbox{$\kappa(\mat{A}_h)$ vs $\epsilon$}
{
	\begin{tikzpicture}[scale=0.57]
	\begin{loglogaxis}[
	  legend columns=3, 
	  legend to name=aggr_cut_cells, 
	  xlabel={$\epsilon$}, 
	  ylabel={$\kappa(\mat{A}_h)$}
	  ]
	\addplot table[x=Epsilon,y=Condition] {data/aggr_cut_cells_test_vol_stab_k0.dat};
	\addlegendentry{\(k=0\)};
	\addplot table[x=Epsilon,y=Condition] {data/aggr_cut_cells_test_vol_stab_k1.dat};
	\addlegendentry{\(k=1\)};
	\addplot table[x=Epsilon,y=Condition] {data/aggr_cut_cells_test_vol_stab_k2.dat};
	\addlegendentry{\(k=2\)};
	\end{loglogaxis}
	\end{tikzpicture}
}
\subcaptionbox{$\lambda_{\rm{min}}(\mat{A}_h)$ vs $\epsilon$}
{
	\begin{tikzpicture}[scale=0.57]
	\begin{loglogaxis}[
	  legend columns=3, 
	  xlabel={$\epsilon$}, 
	  ylabel={$\lambda_{\rm{min}}(\mat{A}_h)$},
	  ymin=0.5,
	  ymax=0.6,
	  ytick={0.5,0.55,0.6}, 
	  yticklabels={0.5,0.55,0.6}
	  ]
	\addplot table[x=Epsilon,y=MinEig] {data/aggr_cut_cells_test_vol_stab_k0.dat};
	\addplot table[x=Epsilon,y=MinEig] {data/aggr_cut_cells_test_vol_stab_k1.dat};
	\addplot table[x=Epsilon,y=MinEig] {data/aggr_cut_cells_test_vol_stab_k2.dat};
	\end{loglogaxis}
	\end{tikzpicture}
}
\subcaptionbox{$\lambda_{\rm{max}}(\mat{A}_h)$ vs $\epsilon$}
{
	\begin{tikzpicture}[scale=0.57]
	\begin{loglogaxis}[legend columns=3, xlabel={$\epsilon$}, ylabel={$\lambda_{\rm{max}}(\mat{A}_h)$}]
	\addplot table[x=Epsilon,y=MaxEig] {data/aggr_cut_cells_test_vol_stab_k0.dat};
	\addplot table[x=Epsilon,y=MaxEig] {data/aggr_cut_cells_test_vol_stab_k1.dat};
	\addplot table[x=Epsilon,y=MaxEig] {data/aggr_cut_cells_test_vol_stab_k2.dat};
	\end{loglogaxis}
	\end{tikzpicture}
}
\caption{Cut meshes with aggregated elements}
\label{fig:aggr.cells}
\end{figure}

\subsubsection{Test B}

We now consider a sequence of cut and approximate meshes of the circular domain $\Omega = \{(x, y)\in\R^2:x^2 + y^2 < 1\}$ and track the conditioning of the scheme before and after the agglomeration of sliver-cut and small-cut elements. The parameters of this sequence of meshes are given in Table \ref{table:mesh.circular.no.aggr} and three of the meshes are plotted in Figure \ref{fig:mesh.circular.no.aggr}.

\begin{table}[H]
	\centering
	\pgfplotstableread{data/circular_tri_no_aggr_k0.dat}\loadedtable
	\pgfplotstabletypeset
	[
	columns={hMin,hMax,NbCells,NbInternalEdges}, 
	columns/hMin/.style={column name=\(\hmin\)},
	columns/hMax/.style={column name=\(\hmax\)},
	columns/NbCells/.style={column name=Nb. Elements},
	columns/NbInternalEdges/.style={column name=Nb. Internal Edges},
	every head row/.style={before row=\toprule,after row=\midrule},
	every last row/.style={after row=\bottomrule} 
	]\loadedtable
	\caption{Parameters of the meshes used in Test B prior to aggregation}
	\label{table:mesh.circular.no.aggr}
\end{table}

\begin{figure}[H]
\centering
\includegraphics[width=0.3\textwidth]{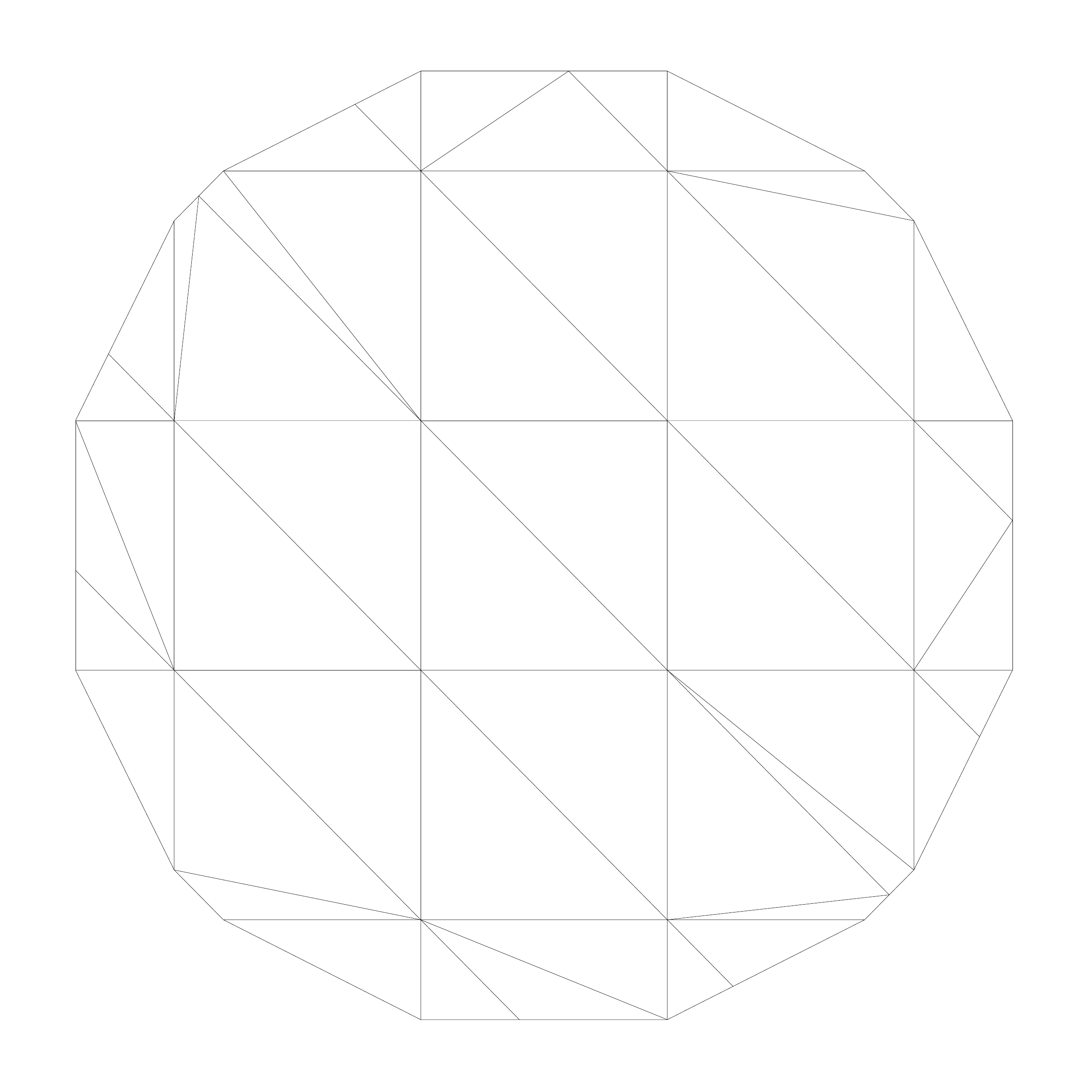} 
\includegraphics[width=0.3\textwidth]{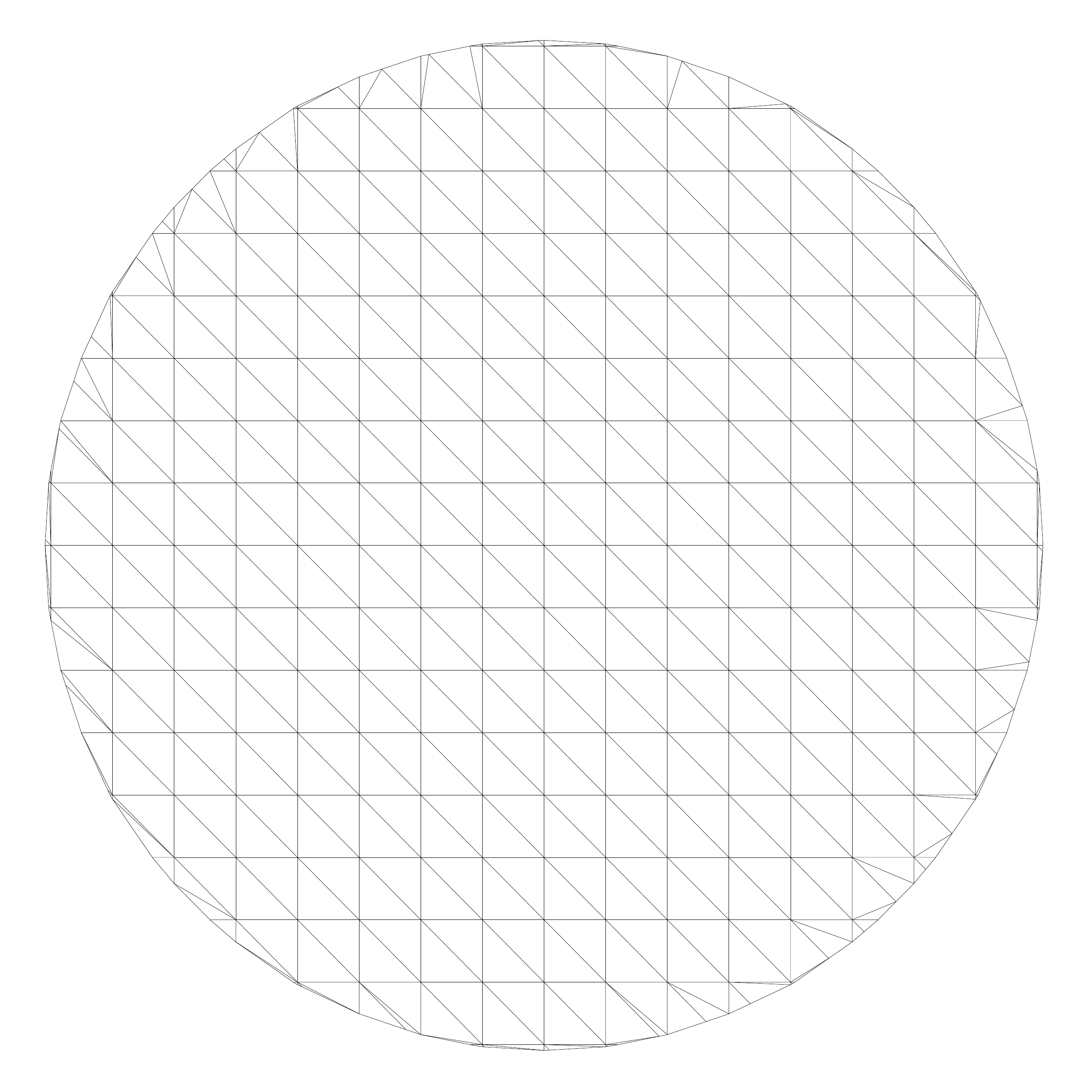} 
\includegraphics[width=0.3\textwidth]{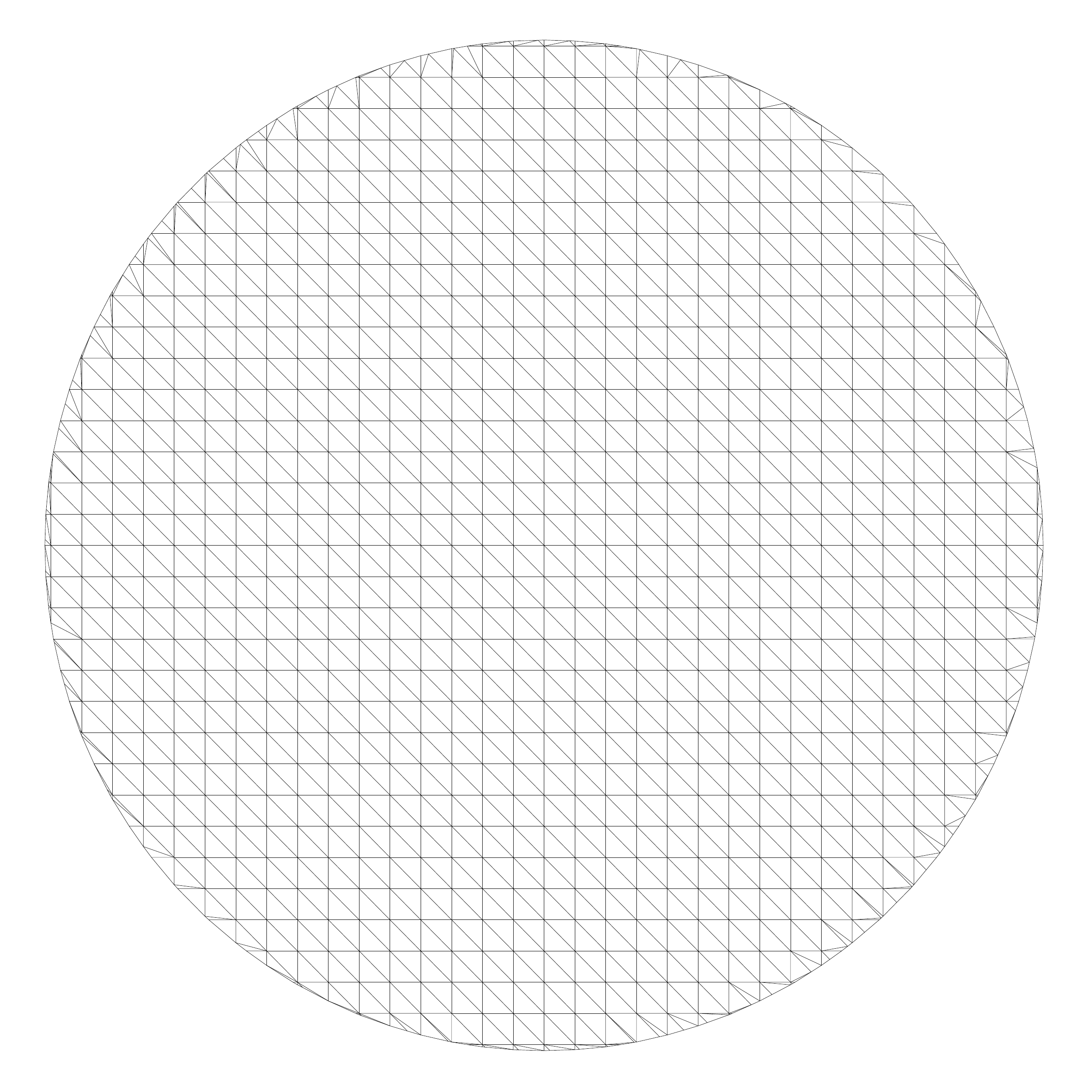} 
\caption{Three of the meshes used in Test B prior to aggregation}
\label{fig:mesh.circular.no.aggr}
\end{figure}

Prior to aggregation of small-cut elements, each face is attached to at least one element of diameter proportional to $\hmax$. Thus we expect to observe $\lambda_{\rm min}(\mat{A}_h) \sim \hmax$ and $\lambda_{\rm max}(\mat{A}_h) \sim \hmin^{-1}$. In Figure \ref{fig:circular.no.aggr} we plot the condition number and eigenvalues for each mesh prior to aggregation. The results are not smooth due to the presence of sliver-cut elements which have potentially very large mesh regularity parameters. In Figure \ref{fig:circular.iso}, we observe that after the agglomeration of sliver-cut elements the results behave as predicted by theory. 

\begin{figure}[H]
\centering
\ref{circular_no_aggr}
\vspace{0.5cm}\\
\subcaptionbox{$\kappa(\mat{A}_h)$ vs $\hmax^{-1}\hmin^{-1}$}
{
	\begin{tikzpicture}[scale=0.57]
	\begin{loglogaxis}[legend columns=3, legend to name=circular_no_aggr, xlabel={$\hmax^{-1}\hmin^{-1}$}, ylabel={$\kappa(\mat{A}_h)$}]
	\addplot table[x=InvMaxMin,y=Condition] {data/circular_tri_no_aggr_vol_stab_k0.dat};
	\addlegendentry{\(k=0\)};
	\addplot table[x=InvMaxMin,y=Condition] {data/circular_tri_no_aggr_vol_stab_k1.dat};
	\addlegendentry{\(k=1\)};
	\addplot table[x=InvMaxMin,y=Condition] {data/circular_tri_no_aggr_vol_stab_k2.dat};
	\addlegendentry{\(k=2\)};
	\logLogSlopeTriangle{0.9}{0.4}{0.1}{1}{black};
	\end{loglogaxis}
	\end{tikzpicture}
}
\subcaptionbox{$\lambda_{\rm{min}}(\mat{A}_h)$ vs $\hmax$}
{
	\begin{tikzpicture}[scale=0.57]
	\begin{loglogaxis}[legend columns=3, xlabel={$\hmax$}, ylabel={$\lambda_{\rm{min}}(\mat{A}_h)$}]
	\addplot table[x=hMax,y=MinEig] {data/circular_tri_no_aggr_vol_stab_k0.dat};
	\addplot table[x=hMax,y=MinEig] {data/circular_tri_no_aggr_vol_stab_k1.dat};
	\addplot table[x=hMax,y=MinEig] {data/circular_tri_no_aggr_vol_stab_k2.dat};
	\logLogSlopeTriangle{0.9}{0.4}{0.1}{1}{black};
	\end{loglogaxis}
	\end{tikzpicture}
}
\subcaptionbox{$\lambda_{\rm{max}}(\mat{A}_h)$ vs $\hmin^{-1}$}
{
	\begin{tikzpicture}[scale=0.57]
	\begin{loglogaxis}[legend columns=3, xlabel={$\hmin^{-1}$}, ylabel={$\lambda_{\rm{max}}(\mat{A}_h)$}]
	\addplot table[x=hMinInv,y=MaxEig] {data/circular_tri_no_aggr_vol_stab_k0.dat};
	\addplot table[x=hMinInv,y=MaxEig] {data/circular_tri_no_aggr_vol_stab_k1.dat};
	\addplot table[x=hMinInv,y=MaxEig] {data/circular_tri_no_aggr_vol_stab_k2.dat};
	\logLogSlopeTriangle{0.9}{0.4}{0.1}{1}{black};
	\end{loglogaxis}
	\end{tikzpicture}
}
\caption{Circular meshes with no aggregation}
\label{fig:circular.no.aggr}
\end{figure}
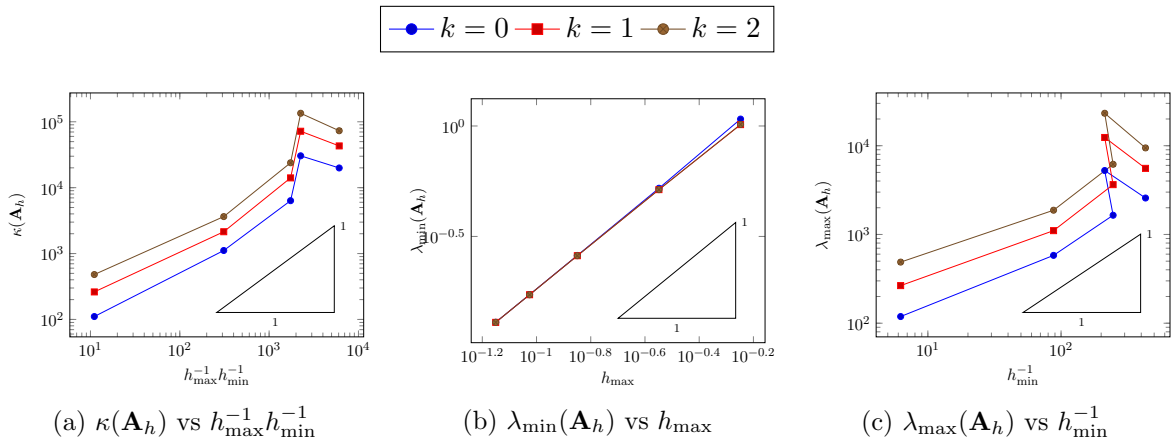

\begin{figure}[H]
\centering
\ref{circular_iso}
\vspace{0.5cm}\\
\subcaptionbox{$\kappa(\mat{A}_h)$ vs $h_{\textrm{min}}^{-1}h_{\textrm{max}}^{-1}$}
{
	\begin{tikzpicture}[scale=0.57]
	\begin{loglogaxis}[legend columns=3, legend to name=circular_iso, xlabel={$h_{\textrm{min}}^{-1}h_{\textrm{max}}^{-1}$}, ylabel={$\kappa(\mat{A}_h)$}]
	\addplot table[x=InvMaxMin, y=Condition] {data/circular_tri_isometric_vol_stab_k0.dat};
	\addlegendentry{\(k=0\)};
	\addplot table[x=InvMaxMin,y=Condition] {data/circular_tri_isometric_vol_stab_k1.dat};
	\addlegendentry{\(k=1\)};
	\addplot table[x=InvMaxMin,y=Condition] {data/circular_tri_isometric_vol_stab_k2.dat};
	\addlegendentry{\(k=2\)};
	\logLogSlopeTriangle{0.9}{0.4}{0.1}{1}{black};
	\end{loglogaxis}
	\end{tikzpicture}
}
\subcaptionbox{$\lambda_{\rm{min}}(\mat{A}_h)$ vs $h_{\textrm{max}}$}
{
	\begin{tikzpicture}[scale=0.57]
	\begin{loglogaxis}[legend columns=3, xlabel={$h_{\textrm{max}}$}, ylabel={$\lambda_{\rm{min}}(\mat{A}_h)$}]
	\addplot table[x=hMax,y=MinEig] {data/circular_tri_isometric_vol_stab_k0.dat};
	\addplot table[x=hMax,y=MinEig] {data/circular_tri_isometric_vol_stab_k1.dat};
	\addplot table[x=hMax,y=MinEig] {data/circular_tri_isometric_vol_stab_k2.dat};
	\logLogSlopeTriangle{0.9}{0.4}{0.1}{1}{black};
	\end{loglogaxis}
	\end{tikzpicture}
}
\subcaptionbox{$\lambda_{\rm{max}}(\mat{A}_h)$ vs $h_{\textrm{min}}^{-1}$}
{
	\begin{tikzpicture}[scale=0.57]
	\begin{loglogaxis}[legend columns=3, xlabel={$h_{\textrm{min}}^{-1}$}, ylabel={$\lambda_{\rm{max}}(\mat{A}_h)$}]
	\addplot table[x=hMinInv,y=MaxEig] {data/circular_tri_isometric_vol_stab_k0.dat};
	\addplot table[x=hMinInv,y=MaxEig] {data/circular_tri_isometric_vol_stab_k1.dat};
	\addplot table[x=hMinInv,y=MaxEig] {data/circular_tri_isometric_vol_stab_k2.dat};
	\logLogSlopeTriangle{0.9}{0.4}{0.1}{1}{black};
	\end{loglogaxis}
	\end{tikzpicture}
}
\caption{Circular meshes with sliver-cut elements aggregated}
\label{fig:circular.iso}
\end{figure}
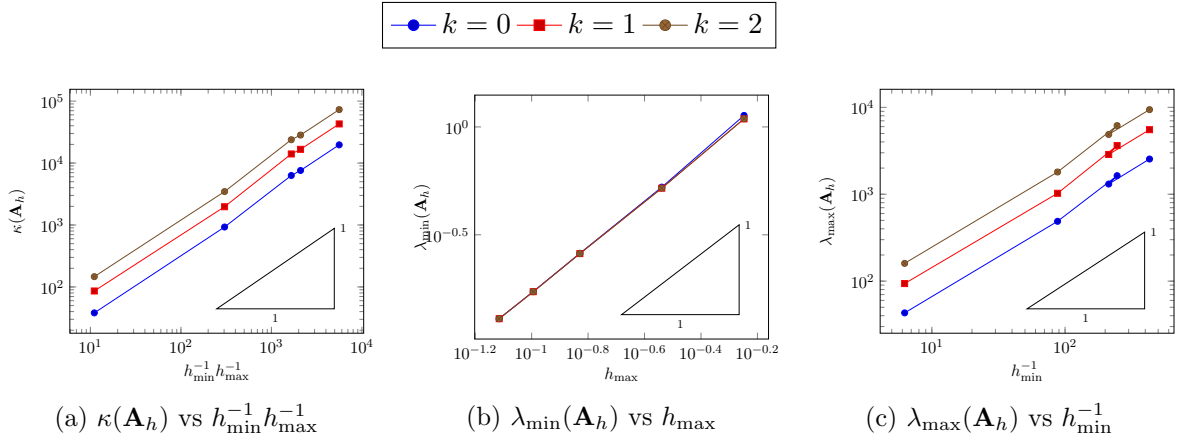

In Figure \ref{fig:circular.iso.uniform}, results are plotted with both sliver-cut and small-cut elements aggregated. The condition number is one order of magnitude smaller than it was prior to aggregation, and scales as $h^{-2}$. Again, this is expected due to the meshes being quasi-uniform ($h=\hmax\sim\hmin$).

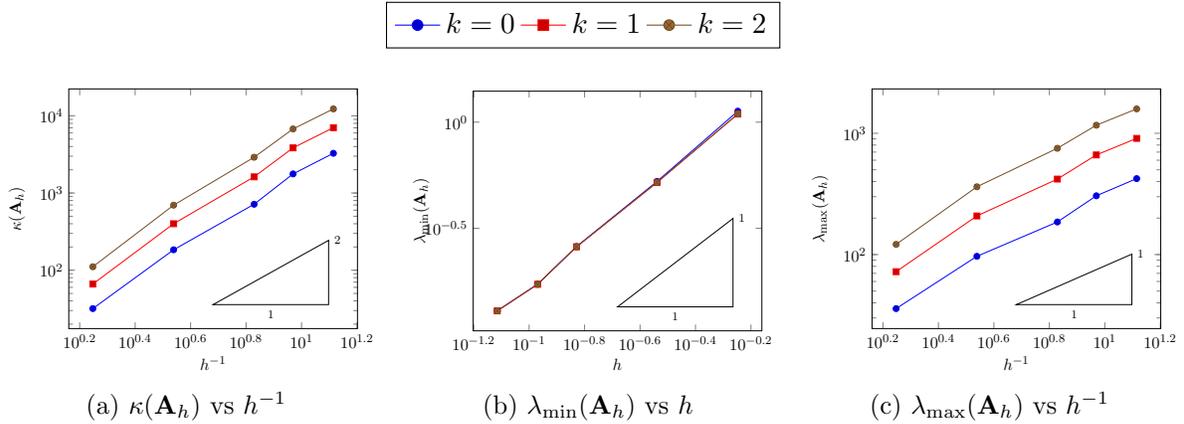
\begin{figure}[H]
	\centering
	\ref{circular_iso_uniform}
	\vspace{0.5cm}\\
	\subcaptionbox{$\kappa(\mat{A}_h)$ vs $h^{-1}$}
	{
		\begin{tikzpicture}[scale=0.56]
		\begin{loglogaxis}[legend columns=3, legend to name=circular_iso_uniform, xlabel={$h^{-1}$}, ylabel={$\kappa(\mat{A}_h)$}]
		\addplot table[x=hMaxInv,y=Condition] {data/circular_tri_isometric_hom_vol_stab_k0.dat};
		\addlegendentry{\(k=0\)};
		\addplot table[x=hMaxInv,y=Condition] {data/circular_tri_isometric_hom_vol_stab_k1.dat};
		\addlegendentry{\(k=1\)};
		\addplot table[x=hMaxInv,y=Condition] {data/circular_tri_isometric_hom_vol_stab_k2.dat};
		\addlegendentry{\(k=2\)};
		\logLogSlopeTriangle{0.9}{0.4}{0.1}{2}{black};
		\end{loglogaxis}
		\end{tikzpicture}
	}
	\subcaptionbox{$\lambda_{\rm{min}}(\mat{A}_h)$ vs $h$}
	{
		\begin{tikzpicture}[scale=0.56]
		\begin{loglogaxis}[legend columns=3, xlabel={$h$}, ylabel={$\lambda_{\rm{min}}(\mat{A}_h)$}]
		\addplot table[x=hMax,y=MinEig] {data/circular_tri_isometric_hom_vol_stab_k0.dat};
		\addplot table[x=hMax,y=MinEig] {data/circular_tri_isometric_hom_vol_stab_k1.dat};
		\addplot table[x=hMax,y=MinEig] {data/circular_tri_isometric_hom_vol_stab_k2.dat};
		\logLogSlopeTriangle{0.9}{0.4}{0.1}{1}{black};
		\end{loglogaxis}
		\end{tikzpicture}
	}
	\subcaptionbox{$\lambda_{\rm{max}}(\mat{A}_h)$ vs $h^{-1}$}
	{
		\begin{tikzpicture}[scale=0.56]
		\begin{loglogaxis}[legend columns=3, xlabel={$h^{-1}$}, ylabel={$\lambda_{\rm{max}}(\mat{A}_h)$}]
		\addplot table[x=hMaxInv,y=MaxEig] {data/circular_tri_isometric_hom_vol_stab_k0.dat};
		\addplot table[x=hMaxInv,y=MaxEig] {data/circular_tri_isometric_hom_vol_stab_k1.dat};
		\addplot table[x=hMaxInv,y=MaxEig] {data/circular_tri_isometric_hom_vol_stab_k2.dat};
		\logLogSlopeTriangle{0.9}{0.4}{0.1}{1}{black};
		\end{loglogaxis}
		\end{tikzpicture}
	}
	\caption{Circular meshes with sliver-cut and small-cut elements aggregated}
	\label{fig:circular.iso.uniform}
\end{figure}

\subsection{Penta-diagonal meshes}

We consider in this section a family of meshes with a penta-diagonal of elements being refined, and two large elements
on each side (see Figure \ref{fig:penta.mesh}). The purpose of this test is to assess the accuracy of our estimates, and the robustness of the HHO
condition number itself, when some large elements are neighbouring very small elements, all the while having an increasing number of
faces. While testing on such extreme meshes is possibly contrived, the behaviour of the condition number illustrates that in some situations the estimates of Theorem \ref{th:estimates} can be improved.

\begin{figure}[H]
	\centering
	\includegraphics[width=3\textwidth/10]{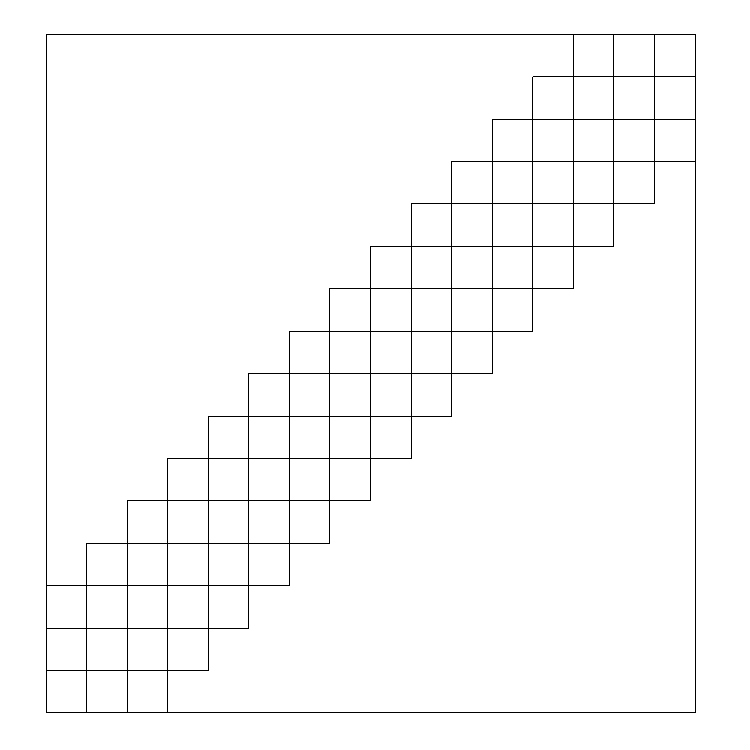} 
	\includegraphics[width=3\textwidth/10]{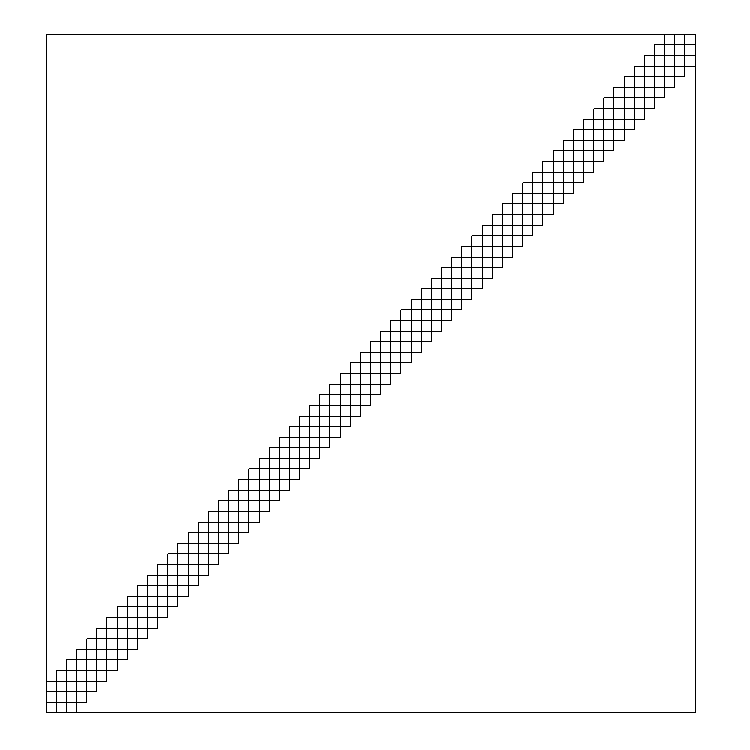} 
	\caption{Penta-diagonal square meshes}\label{fig:penta.mesh}
\end{figure}

\begin{figure}[H]
	\centering
	\ref{square_penta_diag_no_condition}
	\vspace{0.5cm}\\
	\subcaptionbox{$\kappa(\mat{A}_h)$ vs $h_{\textrm{min}}^{-1}$\label{fig:penta.kappa}}
	{
		\begin{tikzpicture}[scale=0.57]
		\begin{loglogaxis}[legend columns=3, legend to name=square_penta_diag_no_condition, xlabel={$h_{\textrm{min}}^{-1}$}, ylabel={$\kappa(\mat{A}_h)$}]
		\addplot table[x=hMinInv, y=Condition] {data/square_pentadiag_vol_stab_k0.dat};
		\addlegendentry{\(k=0\)};
		\addplot table[x=hMinInv,y=Condition] {data/square_pentadiag_vol_stab_k1.dat};
		\addlegendentry{\(k=1\)};
		\addplot table[x=hMinInv,y=Condition] {data/square_pentadiag_vol_stab_k2.dat};
		\addlegendentry{\(k=2\)};
		\logLogSlopeTriangle{0.9}{0.4}{0.1}{1}{black};
		\logLogSlopeTriangle{0.9}{0.4}{0.1}{2}{black};
		\end{loglogaxis}
		\end{tikzpicture}
	}
	\subcaptionbox{$\lambda_{\rm{min}}(\mat{A}_h)$ vs $h_{\textrm{min}}^{-1}$\label{fig:penta.lambdamin}}
	{
		\begin{tikzpicture}[scale=0.57]
		\begin{loglogaxis}[legend columns=3, xlabel={$h_{\textrm{min}}^{-1}$}, ylabel={$\lambda_{\rm{min}}(\mat{A}_h)$}]
		\addplot table[x=hMinInv,y=MinEig] {data/square_pentadiag_vol_stab_k0.dat};
		\addplot table[x=hMinInv,y=MinEig] {data/square_pentadiag_vol_stab_k1.dat};
		\addplot table[x=hMinInv,y=MinEig] {data/square_pentadiag_vol_stab_k2.dat};
		\end{loglogaxis}
		\end{tikzpicture}
	}
	\subcaptionbox{$\lambda_{\rm{max}}(\mat{A}_h)$ vs $h_{\textrm{min}}^{-1}$}
	{
		\begin{tikzpicture}[scale=0.57]
		\begin{loglogaxis}[legend columns=3, xlabel={$h_{\textrm{min}}^{-1}$}, ylabel={$\lambda_{\rm{max}}(\mat{A}_h)$}]
		\addplot table[x=hMinInv,y=MaxEig] {data/square_pentadiag_vol_stab_k0.dat};
		\addplot table[x=hMinInv,y=MaxEig] {data/square_pentadiag_vol_stab_k1.dat};
		\addplot table[x=hMinInv,y=MaxEig] {data/square_pentadiag_vol_stab_k2.dat};
		\logLogSlopeTriangle{0.9}{0.4}{0.1}{1}{black};
		\end{loglogaxis}
		\end{tikzpicture}
	}
	\caption{Penta-diagonal square meshes}\label{fig:penta.results}
\end{figure}

The results presented in Figure \ref{fig:penta.results} show a growth of the maximum eigenvalue as $\mathcal O(h_{\textrm{min}}^{-1})$, which is consistent with the estimate \eqref{est:lambda.max}. Figure \ref{fig:penta.lambdamin} however seems to indicate that, for this family of meshes, $\lambda_{\rm min}(\mat{A}_h)$ actually remains bounded below, which would indicate that the estimate \eqref{est:lambda.min} is not optimal; it can actually be proved (see Lemma \ref{lem:lmin.penta}) that for these meshes the minimal eigenvalue indeed remains bounded below. As a consequence, the condition number $\kappa(\mat{A}_h)$ does not grow as $\mathcal O(h_{\textrm{min}}^{-2})$ but as $\mathcal O(h_{\textrm{min}}^{-1})$, which is illustrated in Figure \ref{fig:penta.kappa}.

\begin{lemma}\label{lem:lmin.penta}
	For the family of penta-diagonal meshes, it holds that
	$\lambda_{\rm min}(\mat{A}_h)\gtrsim 1$.	
\end{lemma}

\begin{proof}
We first note that even if the penta-diagonal meshes do not satisfy Assumption \ref{assum:star.shaped} (due to the two large elements with ``stairs'' boundary), the analysis carried out in the previous sections still applies. Indeed, we can easily find uniform bi-Lipschitz mappings between each of these elements and a ball of size comparable to these elements, which ensures that the trace inequality \eqref{eq:continuous.trace} still holds; since all elements contain a ball of size comparable to their diameters, the other relevant inequalities (approximation properties of projectors, discrete inverse inequalities) also remain valid.

An inspection of the proof of Theorem \ref{th:estimates} (see in particular \eqref{eq:Ah.lower.bound}) reveals that the bound on $\lambda_{\rm min}(\mat{A}_h)$ is a direct consequence of \eqref{eq:discrete.poincare.faces}. The result thus follows if we establish this improved version of \eqref{eq:discrete.poincare.faces}, in which the scaling factor $h_T$ has been removed from the left-hand side: for all $\ulvh\in\Uhklzr$,
\begin{equation}\label{eq:discrete.poincare.faces.penta}
	\sum_{T\in\Th}\norm[\bdryT]{\vFT}^2 \lesssim \sum_{T\in\Th}\Brac{\norm[T]{\nabla\pT{k+1}\ulvT}^2 + \hT^{-1}\norm[\bdryT]{\deltaFT{k}\ulvT}^2}.
\end{equation}

Let us take a face $F$ in one of the small elements. Assuming for example that $F$ is a vertical face, we can create a finite sequence of vertical faces $(F=F_1,F_2,\ldots,F_r)$ (with $r\le 3$) such that $F_r$ is a face of one of the two big elements in the mesh, say $T_\star$; see Figure \ref{fig:proof.penta} for an illustration.

\begin{figure}
\begin{center}
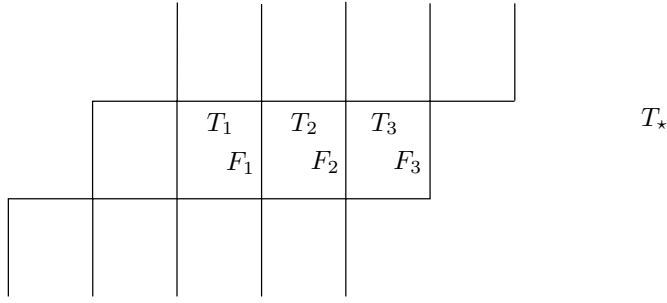
\caption{Illustration of the proof of Lemma \ref{lem:lmin.penta}.}
\label{fig:proof.penta}
\end{center}
\end{figure}

Denoting by $(T_1,\ldots,T_r,T_{r+1}=T_\star)$ the elements encountered along the sequence $(F_1,\ldots,F_r)$, we can then write
\begin{align}\label{eq:penta.proof.1}
	\norm[F]{v_F}^2 \le{}& \norm[F]{v_{F_1}-\pi_{F_{1}}^{0, k}\rmp_{T_2}^{k+1}\ul{v}_{T_2}}^2 + \norm[F]{\rmp_{T_2}^{k+1}\ul{v}_{T_2} - \pi_{T_{2}}^{0, 0}\rmp_{T_2}^{k+1}\ul{v}_{T_2}}^2 + \norm[F]{\pi_{T_2}^{0,0}\rmp_{T_2}^{k+1}\ul{v}_{T_2}}^2 \nl
	\les \norm[\bdry T_2]{\delta_{\bdry T_2}^{k}\ul{v}_{T_2}}^2 + h_{T_2}\norm[T_2]{\nabla \rmp_{T_2}^{k+1}\ul{v}_{T_2}}^2 + \norm[F_2]{\pi_{T_2}^{0,0}\rmp_{T_2}^{k+1}\ul{v}_{T_2}}^2
\end{align}
where we have introduced $\pm\pi_{F_{1}}^{0, k}(\rmp_{T_2}^{k+1}\ul{v}_{T_2}-\pi_{T_2}^{0,0}\rmp_{T_2}^{k+1}\ul{v}_{T_2})=\pm\pi_{F_{1}}^{0, k}\rmp_{T_2}^{k+1}\ul{v}_{T_2}-\pi_{T_2}^{0,0}\rmp_{T_2}^{k+1}\ul{v}_{T_2}$ and used the $L^2(F)$-boundedness of $\pi_{F_{1}}^{0, k}$ and a triangle inequality in the first line, and invoked in the second line the bound $\norm[F_1]{\delta_{F_1}^{k}\ul{v}_{T_2}}^2 \le \norm[\bdry T_2]{\delta_{\bdry T_2}^{k}\ul{v}_{T_2}}^2$, the continuous trace inequality \eqref{eq:continuous.trace} and Poincar\'{e}--Wirtinger inequality \eqref{eq:poincare}, and the fact that $\pi_{T_2}^{0,0}\rmp_{T_2}^{k+1}\ul{v}_{T_2}$ is constant and $|F|=|F_2|$. By a similar argument it holds that
\begin{align}\label{eq:penta.proof.2}
	\norm[F_2]{\pi_{T_2}^{0,0}\rmp_{T_2}^{k+1}\ul{v}_{T_2}}^2 \le{}& \norm[F_2]{\pi_{T_2}^{0,0}\rmp_{T_2}^{k+1}\ul{v}_{T_2} - \pi_{F_2}^{0,k}\rmp_{T_2}^{k+1}\ul{v}_{T_2}}^2 + \norm[F_2]{\pi_{F_2}^{0,k}\rmp_{T_2}^{k+1}\ul{v}_{T_2} - v_{F_2}}^2 + \norm[F_2]{v_{F_2}}^2\nl
	\les h_{T_2}\norm[T_2]{\nabla \rmp_{T_2}^{k+1}\ul{v}_{T_2}}^2 +\norm[\bdry T_2]{\delta_{\bdry T_2}^{k}\ul{v}_{T_2}}^2 + \norm[F_2]{v_{F_2}}^2.
\end{align}
Thus, combining \eqref{eq:penta.proof.1} and \eqref{eq:penta.proof.2} we are able to write
\[
	\norm[F_1]{v_{F_1}}^2 \lesssim h_{T_2}\norm[T_2]{\nabla \rmp_{T_2}^{k+1}\ul{v}_{T_2}}^2 +\norm[\bdry T_2]{\delta_{\bdry T_2}^{k}\ul{v}_{T_2}}^2+ \norm[F_2]{v_{F_2}}^2.
\]
Iterating these estimates along the family $(F_1,F_2,\ldots,F_r)$, using $r\le 3$ and $h_{T_i}\lesssim 1$ we deduce that
\[
	\norm[F]{v_F}^2\lesssim \sum_{i=2}^r \left(\norm[T_i]{\nabla\pTi{k+1}\ulvTi}^2 + \hTi^{-1}\norm[\bdryTi]{\deltaFTi{k}\ulvTi}^2\right)+\norm[F_r]{v_{F_r}}^2.
\]
Summing this inequality over $F\in\Fh[T]$ and then over the small elements $T$ on the diagonal of the mesh, each of the small diagonal elements
appear at most $3$ times in the right-hand side, and the last boundary term is bounded above by $\norm[\partial T_\star]{v_{\partial T_\star}}^2$. This term can be estimated by introducing $\pi_{\bdry T_\star}^{0,k}\rmp^{k+1}_{T_\star}\ul{v}_{T_\star}$ as follows:
\begin{align*}
\norm[\partial T_\star]{v_{\partial T_\star}}^2\lesssim{}& \norm[\partial T_\star]{v_{\partial T_\star}-\pi_{\bdry T_\star}^{0,k}\rmp^{k+1}_{T_\star}\ul{v}_{T_\star}}^2+
\norm[\partial T_\star]{\pi_{\bdry T_\star}^{0,k}\rmp^{k+1}_{T_\star}\ul{v}_{T_\star}}^2\\
\lesssim{}& \norm[\bdryTs]{\deltaFTs{k}\ulvTs}^2 +h_{T_\star}\norm[T_i]{\nabla\pTs{k+1}\ulvTs}^2+ 
h_{T_\star}^{-1}\norm[T_\star]{\rmp^{k+1}_{T_\star}\ul{v}_{T_\star}}^2,\\ 
\lesssim{}& \hTs^{-1}\norm[\bdryTs]{\deltaFTs{k}\ulvTs}^2 +\norm[T_i]{\nabla\pTs{k+1}\ulvTs}^2+ \norm[T_\star]{\rmp^{k+1}_{T_\star}\ul{v}_{T_\star}}^2,
\end{align*}
where we have invoked the boundedness of $\pi_{\bdry T_\star}^{0,k}$ and the continuous trace inequality \eqref{eq:continuous.trace}, and in the final line we have used $h_{T_\star}\approx 1$, since $T_\star$ is one of the large elements whose diameter does not go to zero. Combining all these estimates leads to
$$
\sum_{T\in\Th}\norm[\bdryT]{\vFT}^2\lesssim \sum_{T\in\Th}\left(\norm[T]{\nabla\pT{k+1}\ulvT}^2 + \hT^{-1}\norm[\bdryT]{\deltaFT{k}\ulvT}^2\right) + \norm[\Omega]{\ph{k+1}\ulvh}^2.
$$
The estimate \eqref{eq:discrete.poincare.faces.penta} then follows in the same manner as in the proof of \eqref{eq:discrete.poincare.faces}.
\end{proof}

\section{Conclusions}\label{sec:conclusions}

In this work, we prove detailed eigenvalue and condition number bounds for the linear system matrix that arises from \ac{hho} discretisations of the Laplace problem. The analysis applies to general polytopal meshes and polynomial orders. It reveals the effect of small and highly distorted elements and faces on the conditioning of the linear system. Whereas highly distorted elements negatively impact condition numbers, faces shapes and sizes do not affect these bounds. With this information, we apply \ac{hho} methods on cut meshes. We combine simple background meshes, element intersection algorithms and an aggregation strategy to end up with well-posed \ac{hho} methods on cut meshes.

We carry out a detailed set of numerical experiments that are in agreement with the numerical analysis. First, we analyse the condition number as one coarsens polytopal meshes with many faces per element and arbitrarily small faces. Next, we show that the \ac{hho} method on aggregated cut meshes provides the expected condition number with respect to the mesh size. We also observe that the condition number of the algorithm is not affected by increasingly small cut elements. Finally, we consider a limit case with penta-diagonal squared meshes that motivates sharper condition number bounds for some specific mesh configurations.

Future work includes the combination of \ac{hho} methods with higher-order cut geometrical discretisations (curved boundaries) and the design of optimal and scalable preconditioners for these linear systems.

\section*{Declarations}

{\footnotesize{
\textbf{Funding} This work was partially supported by the Australian Government through the \emph{Australian Research Council}'s Discovery Projects funding scheme (grant number DP210103092). 

\medskip

\noindent\textbf{Competing Interests} The corresponding author states on behalf of all authors, that there is no conflict of interest. 

\medskip

\noindent\textbf{Code Availability} All code used in this article is available in open source libraries and duly cited. 

\medskip

\noindent\textbf{Data Availability} The data generated in this article is available from the corresponding author on reasonable request.}}

\bibliographystyle{plain}
\bibliography{references}

\end{document}